\documentclass[12pt]{amsart}
\usepackage{amsmath}
\usepackage{amsfonts}
\usepackage{graphics}
\usepackage{epsfig}
\usepackage{amssymb}
\usepackage{amscd}
\usepackage[all]{xy}
\usepackage{latexsym}
\usepackage{graphicx}
\usepackage{multirow}
\usepackage{geometry}

\theoremstyle{plain}
\newtheorem{thm}{Theorem}[section]
\newtheorem{lem}[thm]{Lemma}
\newtheorem{cor}[thm]{Corollary}
\newtheorem{prop}[thm]{Proposition}

\theoremstyle{definition}
\newtheorem{defn}[thm]{Definition}

\theoremstyle{remark}
\newtheorem*{rem}{Remark}

\newtheorem{ass}[thm]{Assumption}

\numberwithin{equation}{section} \numberwithin{figure}{section}

\renewcommand*{\to}{\rightarrow}

\newcommand{\mb}[1]{\mathbb{#1}} 

\newcommand{\mc}[1]{\mathcal{#1}}

\newcommand{\id}{\operatorname{id}}

\usepackage{url}
\usepackage{pgfplots}
\pgfplotsset{compat=1.18}
\usetikzlibrary{calc} 
\usepackage{mathtools}
\usepackage{tikz-cd}

\title[Weyl Curves and Zeta Determinants]
{Weyl Curves and Zeta Determinants of Conic Laplacians
on Riemann Surfaces}

\begin{document}

\author{Jia-Ming (Frank) Liou}
\address{Department of Mathematics\\
National Cheng Kung University\\
No.1, University Road, Tainan City 701, Taiwan\\ fjmliou@mail.ncku.edu.tw}

\begin{abstract}
We study self-adjoint realizations of conic Laplacians on compact
Riemann surfaces with radial conic metrics whose cone angles are
integral multiples of \(2\pi\). Their critical asymptotic coefficients
span a finite-dimensional space with a nondegenerate skew-Hermitian
Green form, whose Lagrangian subspaces parametrize the self-adjoint
realizations. We
construct the associated Weyl functions and derive Kre\u{\i}n's
resolvent formula, a resolvent trace identity, and, under explicit
zeta-regularity hypotheses, a comparison formula for
positive-spectrum zeta determinants. The boundary data of formal
solutions define a holomorphic Weyl curve in the Grassmannian, with
the Weyl functions as its local graph coordinates. Its real
restriction is a positive curve in the Lagrangian Grassmannian, and
its tangent form is identified with the \(L^2\)-inner product through
the Poisson operator. We further identify
the boundary determinants with the transition functions of the
determinant line bundle induced by the Weyl curve.
\end{abstract}

\maketitle

\section{Introduction}

The study of elliptic operators on spaces with conic singularities
has a long history, tracing back to Weyl's limit-point/limit-circle
analysis of singular Sturm--Liouville operators and to the abstract
theory of self-adjoint extensions of symmetric operators; see
\cite{WH,vnj}. In geometric analysis, conic singularities have been
studied through the spectral theory of singular Riemannian spaces,
beginning with Cheeger's work, and through the analysis of Fuchs-type
differential operators and elliptic cone operators; see, for example,
\cite{cj,cj2,SBW,lm}. In these theories, the leading asymptotic
behavior of functions near the singular set determines a
finite-dimensional symplectic boundary space, and self-adjoint
realizations are described by Lagrangian boundary conditions on this
space; see \cite{gjb}. The geometry and spectra of closed extensions
of elliptic cone operators were studied from a Grassmannian viewpoint
by Gil, Krainer, and Mendoza \cite{GilKrainerMendoza2007}.

In extension theory, Kre\u{\i}n-type formulas express resolvent
differences in terms of Weyl functions; see \cite{dva}. For elliptic
boundary value problems, these functions often appear as
Dirichlet-to-Neumann type operators \cite{bjlm,bjlm2}. Zeta
determinants on conic manifolds were studied in \cite{LP,kk}, while
comparison, gluing, and relative determinant formulas were developed
in \cite{FR,bd,MW}. Determinant lines over Grassmannians of elliptic
boundary conditions and their relation to zeta and boundary Fredholm
determinants were studied by Scott and Wojciechowski
\cite{ScottWojciechowski1997,ScottWojciechowski2000}; see also
\cite{ScottRelativeZeta}.

The work most closely related to the present paper is that of
Hillairet and Kokotov \cite{HL}, who studied Euclidean surfaces with
conical singularities and used Kre\u{\i}n's formula, together with the
associated \(S\)-matrix, to derive comparison formulas for the
zeta-regularized determinants of non-Friedrichs self-adjoint
extensions of the Laplacian. Our approach is
complementary: rather than using the \(S\)-matrix as the main
boundary object, we formulate the resolvent and determinant
comparison problems in terms of a finite-dimensional boundary
symplectic space and its Weyl functions.

We consider radial conic metrics with cone angles \(2\pi n_P\), where
each \(n_P\) is an integer. The critical asymptotic coefficients at
the conic points identify
\(
    D_{\max}(\Delta_X)/
    D_{\min}(\Delta_X)
\)
with a finite-dimensional space \(V_S\), on which the Green form
induces a nondegenerate skew-Hermitian form \(\omega_S\). We give a
direct Fourier-mode construction that fixes the normalization needed
later. Self-adjoint realizations are then parametrized by Lagrangian
subspaces of \((V_S,\omega_S)\). For every Lagrangian pair, we
construct the associated Poisson operator and Weyl function and
derive a finite-rank Kre\u{\i}n resolvent formula. Its trace yields the
logarithmic derivative of a finite-dimensional boundary determinant
\(d_{\mc V,\mc W}^X\). This leads to a comparison formula for
positive-spectrum zeta determinants, including the correction from
the omitted nonpositive spectrum. Since general self-adjoint conic
realizations may have spectral zeta functions with nonstandard
singularities, the determinant comparison is stated for pairs of
realizations satisfying the explicit zeta-regularity hypothesis in
Section~5.

Although the extension-theoretic parametrization, the abstract
Kre\u{\i}n formula, and the general notion of a Weyl function are
standard, their explicit normalized realization for conic critical
asymptotic data allows the same boundary data to encode the resolvent
trace, the finite-dimensional determinant in the zeta comparison, and
the geometry of the Weyl curve developed below.

The Grassmannian geometry of closed cone-operator extensions,
including the holomorphic family obtained from null spaces modulo the
minimal domain, was developed by Gil, Krainer, and Mendoza
\cite{GilKrainerMendoza2007}. The geometric construction used here
arose from organizing the preceding conic analysis in intrinsic
boundary coordinates. After completing this construction, we became
aware of Wang's abstract Weyl curve for a symmetric operator
\cite{WangWeylCurve}. The points of contact include the
Grassmannian-valued boundary data of formal solutions, Weyl functions
as local coordinates, and the associated characteristic bundles.
Whereas Wang emphasizes the abstract classification and complex
geometry of symmetric operators, our focus is the explicitly
normalized critical asymptotic data of conic Laplacians and their
relation to Kre\u{\i}n's resolvent formula, the resolvent trace identity,
zeta-determinant comparison, and the induced determinant line. The
critical boundary data of formal solutions
define a holomorphic Weyl curve
\(
    \mc K_X:
    \Sigma_X\to\operatorname{Gr}_m(V_S).
\)
The Weyl functions are its local graph coordinates. On the real
locus, \(\mc K_X\) is a positive curve in
\(\operatorname{LGr}(V_S,\omega_S)\), whose tangent form is identified
with the \(L^2(X')\)-inner product through the Poisson operator.
Furthermore, we identify \(d_{\mc V,\mc W}^X\) with a scalar
transition function of the determinant line of the pullback
tautological bundle \(\mc E_X\), and show that the corrected zeta
determinants form the local representatives of a smooth section of
\(
    \operatorname{Det}(\mc E_X^{\mb R})
    \otimes
    \mc L_C.
\)

The explicit construction of the critical boundary space is
two-dimensional. Once this space and its Green form are identified,
however, the subsequent Weyl-function, resolvent, and determinant
arguments depend only on finite-dimensional boundary symplectic data
and extend to analogous settings subject to the stated spectral and
zeta-regularity assumptions.

The paper is organized as follows. Section~2 constructs the critical
asymptotic boundary space and the associated Green form for conic
Laplacians on compact Riemann surfaces. Section~3 describes
self-adjoint realizations in terms of Lagrangian boundary conditions
and introduces the corresponding Poisson operators and Weyl
functions. Section~4 proves Kre\u{\i}n's resolvent formula and the
associated trace identity. Section~5 establishes the comparison
formula for positive-spectrum zeta determinants. Section~6 introduces
the Weyl curve, proves the positivity of its real restriction, and
interprets the boundary determinants and corrected zeta determinants
in terms of the associated determinant line bundle.

\section{Conic Laplacians and Boundary Symplectic Spaces}
We construct the boundary symplectic space of a conic Laplacian on a
compact Riemann surface. Although the underlying domain theory is
standard for regular-singular elliptic operators, a direct
Fourier-mode description fixes the notation and normalization needed
below. We use the general cone-domain theory only to identify the
localized remainder and establish uniqueness; the critical modes and
the normalization of their Green pairing are computed explicitly.

Let \(X\) be a compact Riemann surface, let \(S\subset X\) be a finite set,
and set \(X':=X\setminus S\). A conic metric on \(X\) with singular set
\(S\) is a Hermitian metric \(ds_X^2\) on \(X'\) with the following
property: for every \(P\in S\), there exist a neighborhood \(U_P\) of
\(P\), a holomorphic coordinate \(z:U_P\to\mathbb C\) centered at \(P\),
an integer \(n_P\geq 2\), and a smooth positive function
\(\rho_P\in C^\infty(U_P)\) such that
\[
ds_X^2=\rho_P(z)|z|^{2n_P-2}|dz|^2
\]
on \(U_P':=U_P\setminus\{P\}\). The integer \(n_P\) is called the index
at \(P\), and the corresponding cone angle is \(2\pi n_P\). The points
of \(S\) are called the conic points of \(X\).

We further assume that \(ds_X^2\) is radial near each conic point. More
precisely, for every \(P\in S\), there exists a coordinate disk
\(D_P\subset U_P\) centered at \(P\) such that the conformal factor
\(\rho_P(z)\) depends only on \(r=|z|\) on
\(D_P':=D_P\setminus\{P\}\). Thus, we may write
\(
\rho_P(z)=\rho_P(r)
\)
on \(D_P'\). Throughout this section, all conic metrics are assumed to
be radial in this sense.

Let \(dA_X\) denote the area form associated with \(ds_X^2\). The
nonnegative Laplace--Beltrami operator, initially defined on compactly
supported smooth functions, is denoted by
\[
\Delta_{X,c}:C_c^\infty(X')\subset L^2(X',dA_X)
\longrightarrow L^2(X',dA_X).
\]
With respect to a local coordinate \(z\) centered at a conic point
\(P\), it is given by
\[
\Delta_{X,c}
=
-\frac{4}{\rho_P(z)|z|^{2n_P-2}}
\frac{\partial^2}{\partial z\,\partial\overline z}.
\]
Writing \(z=re^{i\theta}\) on \(D_P'\), the radiality assumption gives
\[
\Delta_{X,c}
=
-\frac{1}{\rho_P(r)r^{2n_P}}
\left(
r^2\partial_r^2+r\partial_r+\partial_\theta^2
\right).
\]
Here \(C_c^\infty(X')\) denotes the space of compactly supported smooth
complex-valued functions on \(X'\).

Let \(L^2(X',dA_X)\) denote the Hilbert space of equivalence classes of
measurable complex-valued functions \(u\) on \(X'\) satisfying
\(\|u\|_{L^2(X',dA_X)}^2:=\int_{X'}|u|^2\,dA_X<\infty\), where functions
that agree almost everywhere are identified. Here \(dA_X\) denotes the
area form associated with the metric \(ds_X^2\). We equip
\(L^2(X',dA_X)\) with the Hermitian inner product
\(\langle u,v\rangle_{L^2(X',dA_X)}
:=\int_{X'}u\overline{v}\,dA_X\). For simplicity, we write
\(L^2(X'):=L^2(X',dA_X)\).

Let \(u\in L^2(X')\). We say that the distributional Laplacian of \(u\)
belongs to \(L^2(X')\) if there exists \(F\in L^2(X')\) such that
\(\langle u,\Delta_{X,c}\phi\rangle_{L^2(X')}
=\langle F,\phi\rangle_{L^2(X')}\) for every
\(\phi\in C_c^\infty(X')\). In this case, \(F\) is unique and is denoted
by \(\Delta_Xu\).

We now define the maximal and minimal domains of the Laplacian. The maximal
domain is
\(D_{\max}(\Delta_X):=\{u\in L^2(X'):\Delta_Xu\in L^2(X')\}\), where
\(\Delta_Xu\) is understood in the distributional sense. Equipped with the
graph norm
\(\|u\|_{\Delta_X}:=
\bigl(\|u\|_{L^2(X')}^2+\|\Delta_Xu\|_{L^2(X')}^2\bigr)^{1/2}\),
the space \(D_{\max}(\Delta_X)\) is a Hilbert space.
The minimal domain \(D_{\min}(\Delta_X)\) is the closure of
\(C_c^\infty(X')\) in \(D_{\max}(\Delta_X)\) with respect to the graph
norm \(\|\cdot\|_{\Delta_X}\). Consequently, it is a closed Hilbert
subspace of \(D_{\max}(\Delta_X)\).

We shall identify the finite-dimensional quotient
\(D_{\max}(\Delta_X)/D_{\min}(\Delta_X)\) with the critical
asymptotic boundary data at the conic points. This will also show that
\(D_{\min}(\Delta_X)\subsetneq D_{\max}(\Delta_X)\) whenever
\(S\neq\varnothing\).

\begin{lem}\label{lem:away-from-conic}
Let \(w\in D_{\max}(\Delta_X)\). If
\(\operatorname{supp}w\Subset X'\), then
\(w\in D_{\min}(\Delta_X)\).
\end{lem}
\begin{proof}
Set \(K:=\operatorname{supp}w\), and choose relatively compact open sets
\(\Omega_1,\Omega_2\subset X'\) such that
\(K\Subset\Omega_1\Subset\Omega_2\Subset X'\). Since
\(w\in D_{\max}(\Delta_X)\), both \(w\) and \(\Delta_Xw\) belong to
\(L^2(\Omega_2)\). The interior elliptic estimate gives
\(\|w\|_{H^2(\Omega_1)}
\leq C\bigl(\|\Delta_Xw\|_{L^2(\Omega_2)}
+\|w\|_{L^2(\Omega_2)}\bigr)\), and hence
\(w\in H^2(\Omega_1)\).

Since \(\operatorname{supp}w\Subset\Omega_1\), standard smoothing yields
a sequence \(w_j\in C_c^\infty(\Omega_1)\) such that
\(w_j\to w\) in \(H^2(\Omega_1)\). Since \(\Delta_X\) has smooth
coefficients on \(\Omega_1\), it follows that
\(\Delta_{X,c}w_j\to\Delta_Xw\) in \(L^2(\Omega_1)\). Moreover,
differential operators do not enlarge supports, so \(w_j\), \(w\),
\(\Delta_{X,c}w_j\), and \(\Delta_Xw\) all vanish on
\(X'\setminus\Omega_1\). Consequently, \(w_j\to w\) and
\(\Delta_{X,c}w_j\to\Delta_Xw\) in \(L^2(X')\). Thus \(w_j\to w\)
with respect to the graph norm \(\|\cdot\|_{\Delta_X}\), and therefore
\(w\in D_{\min}(\Delta_X)\).
\end{proof}

\begin{lem}\label{lem:fourier-decomposition}
Fix \(P\in S\), and let \(D_P'\) be a punctured coordinate disk centered
at \(P\) on which
\(ds_X^2=\rho_P(r)r^{2n_P-2}|dz|^2\), where \(z=re^{i\theta}\).
Thus \(dA_X=\rho_P(r)r^{2n_P-1}\,dr\,d\theta\).

For \(u\in L^2(X')\), write its restriction to \(D_P'\) as
\[
    u(r,\theta)=\sum_{j\in\mathbb Z}u_j(r)e^{ij\theta},
\]
where
\[
    u_j(r):=\frac{1}{2\pi}
    \int_0^{2\pi}u(r,\theta)e^{-ij\theta}\,d\theta.
\]
Then, on \(D_P'\),
\[
\Delta_Xu
=
-\sum_{j\in\mathbb Z}
\frac{1}{\rho_P(r)r^{2n_P}}
\left(
r^2u_j''(r)+ru_j'(r)-j^2u_j(r)
\right)e^{ij\theta}
\]
in the sense of distributions, where the derivatives of \(u_j\) are
understood distributionally.
\end{lem}

The assertion follows directly from the standard Fourier decomposition
of the Laplace--Beltrami operator in polar coordinates.

Fix \(P\in S\), and abbreviate \(D':=D_P'\), \(n:=n_P\), and
\(\rho:=\rho_P\). Let \(F\in L^2(D',dA_X)\), and suppose that
\(\Delta_Xu=F\) in the distributional sense for some
\(u\in L^2(D',dA_X)\). Write
\(F(r,\theta)=\sum_{j\in\mathbb Z}F_j(r)e^{ij\theta}\) and
\(u(r,\theta)=\sum_{j\in\mathbb Z}u_j(r)e^{ij\theta}\). Then, for every
\(j\in\mathbb Z\),
\begin{equation}\label{eq:fourier-ode}
r^2u_j''(r)+ru_j'(r)-j^2u_j(r)
=
-\rho(r)r^{2n}F_j(r)
\end{equation}
in the sense of distributions on \((0,1)\).

Set \(f_j(r):=\rho(r)r^{2n-1}F_j(r)\). Since
\(F_j\in L^2((0,1),\rho(r)r^{2n-1}\,dr)\), the Cauchy--Schwarz
inequality gives \(f_j\in L^1(0,1)\). Moreover,
\(f_j\in L^2_{\mathrm{loc}}(0,1)\). It follows from the standard
one-dimensional interior regularity that
\(u_j\in H^2_{\mathrm{loc}}(0,1)\), and
\eqref{eq:fourier-ode} holds almost everywhere.

Fix \(r_0\in(0,1)\). For \(0<r<r_0\), the zero mode has the form
\[
u_0(r)
=
a_0+b_0\log r
-
\int_0^r
\log\frac{r}{\tau}\,f_0(\tau)\,d\tau.
\]
For \(j\neq0\), one initially obtains
\[
\begin{aligned}
u_j(r)
&=
a_jr^{|j|}+b_jr^{-|j|} \\
&\quad
+
\frac{1}{2|j|}
\left(
r^{-|j|}\int_0^r\tau^{|j|}f_j(\tau)\,d\tau
+
r^{|j|}\int_r^{r_0}\tau^{-|j|}f_j(\tau)\,d\tau
\right),
\end{aligned}
\]
where \(a_j,b_j\in\mathbb C\).

Set \(d\mu(r):=\rho(r)r^{2n-1}\,dr\). Weighted Cauchy--Schwarz
estimates show that the integral terms belong to
\(L^2((0,r_0),d\mu)\). Moreover, \(r^{-|j|}\) belongs to this space
if and only if \(|j|<n\). Since
\(u\in L^2(D',dA_X)\), it follows that \(b_j=0\) whenever
\(|j|\geq n\).

If \(1\leq |j|\leq n-1\), another weighted Cauchy--Schwarz estimate
shows that the function
\(\tau\mapsto\tau^{-|j|}f_j(\tau)\) belongs to \(L^1(0,r_0)\).
Absorbing the resulting multiple of \(r^{|j|}\) into
\(a_jr^{|j|}\), we may therefore write
\[
\begin{aligned}
u_j(r)
&=
a_jr^{|j|}+b_jr^{-|j|} \\
&\quad+
\frac{1}{2|j|}
\left(
r^{-|j|}\int_0^r\tau^{|j|}f_j(\tau)\,d\tau
-
r^{|j|}\int_0^r\tau^{-|j|}f_j(\tau)\,d\tau
\right).
\end{aligned}
\]
This normalization separates the two homogeneous critical terms from
the particular solution.

We now return to the study of \(D_{\max}(\Delta_X)\). We first consider
the case in which \(S=\{P\}\) consists of a single conic point. Fix a
coordinate disk \(D_P\Subset U_P\) centered at \(P\) on which the metric
is radial, and write \(z=re^{i\theta}\).

\begin{thm}\label{thm:asymptotic-expansion}
Assume that \(S=\{P\}\). For every \(u\in D_{\max}(\Delta_X)\), there
exist unique coefficients \(a_j,b_j\in\mathbb C\),
\(|j|\leq n_P-1\), such that, on \(D_P':=D_P\setminus\{P\}\),
\[
u(r,\theta)
=
a_0+b_0\log r
+
\sum_{1\leq |j|\leq n_P-1}
\left(
a_jr^{|j|}+b_jr^{-|j|}
\right)e^{ij\theta}
+
R(r,\theta).
\]
Moreover, for every cutoff function \(\chi\in C_c^\infty(D_P)\) equal
to \(1\) in a neighborhood of \(P\), the function \(\chi R\), extended
by zero to \(X'\), belongs to \(D_{\min}(\Delta_X)\).
\end{thm}
\begin{proof}
Set \(F:=\Delta_Xu\in L^2(X')\). After shrinking \(D_P\), if necessary,
we identify it with the unit disk by the coordinate \(z=re^{i\theta}\).
On \(D_P'\), write
\(u(r,\theta)=\sum_{j\in\mathbb Z}u_j(r)e^{ij\theta}\) and
\(F(r,\theta)=\sum_{j\in\mathbb Z}F_j(r)e^{ij\theta}\).

The preceding modewise analysis shows that there exist coefficients
\(a_j,b_j\in\mathbb C\), \(|j|\leq n_P-1\), such that
\[
u(r,\theta)
=
a_0+b_0\log r
+
\sum_{1\leq |j|\leq n_P-1}
\left(
a_jr^{|j|}+b_jr^{-|j|}
\right)e^{ij\theta}
+
R(r,\theta).
\]
Indeed, the \(L^2\)-condition implies that the coefficient of
\(r^{-|j|}e^{ij\theta}\) vanishes whenever \(|j|\geq n_P\).

It remains to identify the localized remainder with an element of the
minimal domain. Near \(P\), the Laplacian is the \(b\)-elliptic cone
operator
\[
    -r^{-2n_P}\rho_P(r)^{-1}
    \bigl((r\partial_r)^2+\partial_\theta^2\bigr)
\]
acting in
\(L^2(\rho_P(r)r^{2n_P}\,dr/r\,d\theta)\). On the angular mode
\(e^{ij\theta}\), the indicial polynomial of the factor in
parentheses is \(\zeta^2-j^2\). In the exponent convention
\(r^\zeta\), the critical strip is
\(-n_P<\operatorname{Re}\zeta<n_P\). Thus its indicial terms are
exactly
\[
    1,\quad \log r,\quad
    r^{|j|}e^{ij\theta},\quad r^{-|j|}e^{ij\theta},
    \qquad 1\leq |j|\leq n_P-1.
\]
The modewise formulas above show that these are precisely the critical
Mellin singular parts of \(u\). To state the resulting estimate for the
remainder precisely, let \(H_b^2\) denote the local \(b\)-Sobolev space
defined using the vector fields \(r\partial_r\) and
\(\partial_\theta\). In the notation of \cite{gjb}, the order is
\(m=2\) and the weight of the present cone operator is
\(\nu=2n_P\). After localization near \(P\),
Proposition~3.6 of \cite{gjb} says that an element
\(w\in D_{\max}(\Delta_X)\) supported in \(D_P\), and extended by zero
to \(X'\), belongs to the minimal domain precisely when
\[
    w\in
    \bigcap_{\varepsilon>0}
    r^{n_P-\varepsilon}H_b^2.
\]
More precisely, Lemma~7.6 and the discussion preceding it in
\cite{gjb} show that removing all the critical Mellin singular parts
leaves a remainder satisfying
\[
    \chi R\in r^{n_P-\varepsilon}H_b^2
    \qquad\text{for every }\varepsilon>0.
\]
Here this is a weighted \(b\)-Sobolev statement, rather than merely a
pointwise estimate. Moreover, \(\chi R\in D_{\max}(\Delta_X)\), since
it is the difference between \(\chi u\) and the finitely many cutoff
critical terms displayed above. Proposition~3.6 therefore gives
\(\chi R\in D_{\min}(\Delta_X)\). The same identification shows that
the critical coefficients are unique.
\end{proof}

Motivated by the preceding asymptotic expansion, we introduce the
following definition.

\begin{defn}
Let \(P\in S\) be a conic point with index \(n_P\), and let \(z_P\) be
the chosen local coordinate centered at \(P\). Write
\(z_P=r_Pe^{i\theta_P}\), where
\(\theta_P\in\mathbb R/2\pi\mathbb Z\). Define
\(f_{P,0}:=1/\sqrt{2\pi}\) and
\(f_{P,0}^{\#}:=-\log r_P/\sqrt{2\pi}\). For
\(1\leq |j|\leq n_P-1\), define
\(f_{P,j}:=r_P^{|j|}e^{ij\theta_P}/\sqrt{4\pi|j|}\) and
\(f_{P,j}^{\#}:=r_P^{-|j|}e^{ij\theta_P}/\sqrt{4\pi|j|}\).

The critical asymptotic space at \(P\), with respect to \(z_P\), is
\[
V_P
:=
\operatorname{span}_{\mathbb C}
\bigl\{
f_{P,j},f_{P,j}^{\#}:|j|\leq n_P-1
\bigr\}.
\]
The elements \(f_{P,j}\) and \(f_{P,j}^{\#}\),
\(|j|\leq n_P-1\), are called the critical asymptotic terms at \(P\).
\end{defn}

\begin{cor}\label{thm:global-decomposition-single-point}
Assume that \(S=\{P\}\), and let \(D_P\) be the coordinate disk fixed
above. For every \(u\in D_{\max}(\Delta_X)\), there exists a unique
element \(h\in V_P\) such that, for every cutoff function
\(\chi\in C_c^\infty(D_P)\) equal to \(1\) near \(P\), the function
\(R_\chi:=u-\chi h\) belongs to \(D_{\min}(\Delta_X)\). Here \(\chi h\)
is extended by zero to \(X'\). Consequently,
\(u=\chi h+R_\chi\), where \(h\in V_P\) and
\(R_\chi\in D_{\min}(\Delta_X)\).
\end{cor}

\begin{proof}
By Theorem~\ref{thm:asymptotic-expansion}, there exists a unique
\(h\in V_P\) such that \(u=h+R_0\) on \(D_P'\), where
\(\chi R_0\in D_{\min}(\Delta_X)\) for every cutoff function \(\chi\)
as above. For any such \(\chi\), one has
\(R_\chi=u-\chi h=\chi R_0+(1-\chi)u\).

The first term belongs to \(D_{\min}(\Delta_X)\) by
Theorem~\ref{thm:asymptotic-expansion}. By interior elliptic regularity
away from \(P\), the second term belongs to \(D_{\max}(\Delta_X)\).
Moreover, it vanishes near \(P\), and hence its support is compactly
contained in \(X'\). Lemma~\ref{lem:away-from-conic} therefore gives
\((1-\chi)u\in D_{\min}(\Delta_X)\). It follows that
\(R_\chi\in D_{\min}(\Delta_X)\).
\end{proof}

The unique element \(h\in V_P\) appearing in the preceding corollary is
called the critical asymptotic part of \(u\) at \(P\).

\begin{cor}\label{cor:smooth-density}
Assume that \(S=\{P\}\). The subspace
\(C^\infty(X')\cap D_{\max}(\Delta_X)\) is dense in
\(D_{\max}(\Delta_X)\) with respect to the graph norm.
\end{cor}

\begin{proof}
Let \(u\in D_{\max}(\Delta_X)\), and let \(h\in V_P\) be its critical
asymptotic part at \(P\). Choose a cutoff function
\(\chi\in C_c^\infty(D_P)\) equal to \(1\) near \(P\), and set
\(R:=u-\chi h\), where \(\chi h\) is extended by zero outside \(D_P\).
By Corollary~\ref{thm:global-decomposition-single-point}, one has
\(R\in D_{\min}(\Delta_X)\).

Since \(D_{\min}(\Delta_X)\) is the graph-norm closure of
\(C_c^\infty(X')\), there exists a sequence
\(R_j\in C_c^\infty(X')\) such that
\(\|R_j-R\|_{\Delta_X}\to0\). Set \(u_j:=\chi h+R_j\). Since
\(\chi h=u-R\in D_{\max}(\Delta_X)\) and \(\chi h\) is smooth on
\(X'\), we have
\(u_j\in C^\infty(X')\cap D_{\max}(\Delta_X)\). Moreover,
\(\|u_j-u\|_{\Delta_X}=\|R_j-R\|_{\Delta_X}\to0\). This proves the
result.
\end{proof}

We now introduce the Green form on \(D_{\max}(\Delta_X)\). For
\(u,v\in D_{\max}(\Delta_X)\), define
\[
q_X(u,v)
:=
\langle\Delta_Xu,v\rangle_{L^2(X')}
-
\langle u,\Delta_Xv\rangle_{L^2(X')}.
\]
Then \(q_X\) is
skew-Hermitian; that is,
\(q_X(v,u)=-\overline{q_X(u,v)}\).

\begin{lem}\label{lem:dmin-kernel}
For every \(u\in D_{\max}(\Delta_X)\) and
\(v\in D_{\min}(\Delta_X)\), one has
\(q_X(u,v)=q_X(v,u)=0\). Consequently, \(q_X\) induces a well-defined
skew-Hermitian sesquilinear form on
\(D_{\max}(\Delta_X)/D_{\min}(\Delta_X)\).
\end{lem}

\begin{proof}
Let \(v\in D_{\min}(\Delta_X)\). By the definition of the minimal
domain, there exists a sequence \(v_j\in C_c^\infty(X')\) such that
\(v_j\to v\) and
\(\Delta_{X,c}v_j\to\Delta_Xv\) in \(L^2(X')\). For every
\(u\in D_{\max}(\Delta_X)\), the definition of the distributional
Laplacian gives
\(\langle\Delta_Xu,v_j\rangle_{L^2(X')}
=\langle u,\Delta_{X,c}v_j\rangle_{L^2(X')}\). Thus
\(q_X(u,v_j)=0\). Passing to the limit yields \(q_X(u,v)=0\).
The skew-Hermitian property then gives \(q_X(v,u)=0\).

Therefore, \(q_X\) is unchanged when either argument is modified by an
element of \(D_{\min}(\Delta_X)\), and hence it induces a
well-defined form on
\(D_{\max}(\Delta_X)/D_{\min}(\Delta_X)\).
\end{proof}

We define the local boundary form
\(\omega_P:V_P\times V_P\to\mathbb C\) by requiring
\(\omega_P(f_{P,j},f_{P,k}^{\#})=\delta_{jk}\),
\(\omega_P(f_{P,j}^{\#},f_{P,k})=-\delta_{jk}\), and
\(\omega_P(f_{P,j},f_{P,k})
=\omega_P(f_{P,j}^{\#},f_{P,k}^{\#})=0\) for
\(|j|,|k|\leq n_P-1\).

Equivalently, if
\(h=\sum_{|j|\leq n_P-1}
(a_jf_{P,j}+b_jf_{P,j}^{\#})\) and
\(\ell=\sum_{|j|\leq n_P-1}
(c_jf_{P,j}+d_jf_{P,j}^{\#})\), then
\[
\omega_P(h,\ell)
=
\sum_{|j|\leq n_P-1}
\left(
a_j\overline{d_j}-b_j\overline{c_j}
\right).
\]

\begin{lem}\label{lem:omega-p-nondegenerate}
The sesquilinear form \(\omega_P\) on \(V_P\) is nondegenerate and
skew-Hermitian.
\end{lem}

\begin{proof}
Let
\(h=\sum_{|j|\leq n_P-1}
(a_jf_{P,j}+b_jf_{P,j}^{\#})\) and
\(\ell=\sum_{|j|\leq n_P-1}
(c_jf_{P,j}+d_jf_{P,j}^{\#})\). By definition,
\(\omega_P(h,\ell)
=\sum_{|j|\leq n_P-1}
(a_j\overline{d_j}-b_j\overline{c_j})\). It follows immediately that
\(\omega_P(\ell,h)=-\overline{\omega_P(h,\ell)}\), so \(\omega_P\) is
skew-Hermitian.

Suppose that \(\omega_P(h,\ell)=0\) for every \(\ell\in V_P\). Taking
\(\ell=f_{P,j}^{\#}\) gives \(a_j=0\), while taking
\(\ell=f_{P,j}\) gives \(b_j=0\), for every
\(|j|\leq n_P-1\). Thus \(h=0\), and hence \(\omega_P\) is
nondegenerate.
\end{proof}

\begin{lem}\label{lem:green-form-local}
Assume that \(S=\{P\}\). Let
\(u,v\in D_{\max}(\Delta_X)\), and let \(h_u,h_v\in V_P\) be their
critical asymptotic parts. Then
\(q_X(u,v)=\omega_P(h_u,h_v)\).
\end{lem}

\begin{proof}
Choose a cutoff function \(\chi\in C_c^\infty(D_P)\) equal to \(1\)
near \(P\). By
Corollary~\ref{thm:global-decomposition-single-point}, write
\(u=\chi h_u+R_u\) and \(v=\chi h_v+R_v\), where
\(R_u,R_v\in D_{\min}(\Delta_X)\). Lemma~\ref{lem:dmin-kernel} gives
\(q_X(u,v)=q_X(\chi h_u,\chi h_v)\).
A direct application of Green's formula to the normalized critical
asymptotic terms gives
\(q_X(\chi f_{P,j},\chi f_{P,k}^{\#})=\delta_{jk}\),
\(q_X(\chi f_{P,j}^{\#},\chi f_{P,k})=-\delta_{jk}\), and
\(q_X(\chi f_{P,j},\chi f_{P,k})
=q_X(\chi f_{P,j}^{\#},\chi f_{P,k}^{\#})=0\) for
\(|j|,|k|\leq n_P-1\). Comparing these identities with the definition
of \(\omega_P\) and using sesquilinearity yields
\(q_X(u,v)=\omega_P(h_u,h_v)\).
\end{proof}

\begin{cor}\label{cor:critical-terms-not-minimal}
Assume that \(S=\{P\}\). Let
\(\chi\in C_c^\infty(D_P)\) be equal to \(1\) near \(P\). Then
\(\chi V_P\cap D_{\min}(\Delta_X)=\{0\}\). Equivalently, if
\(h\in V_P\) is nonzero, then
\(\chi h\notin D_{\min}(\Delta_X)\). In particular,
\(\chi f_{P,j}\) and \(\chi f_{P,j}^{\#}\) do not belong to
\(D_{\min}(\Delta_X)\) for any \(|j|\leq n_P-1\).
\end{cor}

\begin{proof}
Suppose that \(h\in V_P\) and
\(\chi h\in D_{\min}(\Delta_X)\). For every \(\ell\in V_P\), the
function \(\chi\ell\) belongs to \(D_{\max}(\Delta_X)\). Hence
Lemma~\ref{lem:dmin-kernel} gives
\(q_X(\chi h,\chi\ell)=0\). On the other hand,
Lemma~\ref{lem:green-form-local} gives
\(q_X(\chi h,\chi\ell)=\omega_P(h,\ell)\). Thus
\(\omega_P(h,\ell)=0\) for every \(\ell\in V_P\). Since \(\omega_P\)
is nondegenerate, \(h=0\).
\end{proof}

\begin{cor}\label{cor:dmin-proper}
Assume that \(S=\{P\}\). Then
\(D_{\min}(\Delta_X)\subsetneq D_{\max}(\Delta_X)\).
\end{cor}

\begin{proof}
Choose a cutoff function \(\chi\in C_c^\infty(D_P)\) equal to \(1\)
near \(P\). Then \(\chi f_{P,0}\in D_{\max}(\Delta_X)\), whereas
Corollary~\ref{cor:critical-terms-not-minimal} gives
\(\chi f_{P,0}\notin D_{\min}(\Delta_X)\). Hence the inclusion is
strict.
\end{proof}

The left radical of \(q_X\) is defined by
\(\operatorname{rad}_{L}q_X
:=\{u\in D_{\max}(\Delta_X):
q_X(u,v)=0\text{ for every }v\in D_{\max}(\Delta_X)\}\).
Similarly, its right radical is
\(\operatorname{rad}_{R}q_X
:=\{v\in D_{\max}(\Delta_X):
q_X(u,v)=0\text{ for every }u\in D_{\max}(\Delta_X)\}\).
Since \(q_X\) is skew-Hermitian, these two subspaces coincide. We denote
their common value by \(\operatorname{rad}q_X\). By
Lemma~\ref{lem:dmin-kernel}, one has
\(D_{\min}(\Delta_X)\subset\operatorname{rad}q_X\).

\begin{cor}\label{cor:rad-q-dmin}
The radical of \(q_X\) is precisely the minimal domain:
\(\operatorname{rad}q_X=D_{\min}(\Delta_X)\).
\end{cor}

\begin{proof}
Lemma~\ref{lem:dmin-kernel} gives
\(D_{\min}(\Delta_X)\subset\operatorname{rad}q_X\).
Conversely, let \(u\in\operatorname{rad}q_X\), and let \(h_u\in V_P\)
be its critical asymptotic part at \(P\). Choose
\(\chi\in C_c^\infty(D_P)\) equal to \(1\) near \(P\), and set
\(R:=u-\chi h_u\). By
Corollary~\ref{thm:global-decomposition-single-point}, one has
\(R\in D_{\min}(\Delta_X)\).
For every \(\ell\in V_P\), the function \(\chi\ell\) belongs to
\(D_{\max}(\Delta_X)\). Since \(u\in\operatorname{rad}q_X\), one has
\(q_X(u,\chi\ell)=0\). Lemma~\ref{lem:dmin-kernel} and
Lemma~\ref{lem:green-form-local} give
\[
0
=
q_X(u,\chi\ell)
=
q_X(\chi h_u,\chi\ell)
=
\omega_P(h_u,\ell).
\]
Since \(\omega_P\) is nondegenerate, \(h_u=0\). Therefore
\(u=R\in D_{\min}(\Delta_X)\), proving the reverse inclusion.
\end{proof}

\begin{thm}\label{thm:quotient-is-Vp}
Assume that \(S=\{P\}\). Let
\(\pi_P:D_{\max}(\Delta_X)\to V_P\) be the map that assigns to each
\(u\in D_{\max}(\Delta_X)\) its critical asymptotic part at \(P\).
Then \(\pi_P\) is a surjective linear map and
\(\ker\pi_P=D_{\min}(\Delta_X)\). Consequently, \(\pi_P\) induces a
linear isomorphism
\[
\overline{\pi}_P:
D_{\max}(\Delta_X)/D_{\min}(\Delta_X)
\longrightarrow V_P.
\]
In particular,
\(\dim_{\mathbb C}\bigl(D_{\max}(\Delta_X)/
D_{\min}(\Delta_X)\bigr)=4n_P-2\).
\end{thm}

\begin{proof}
The linearity of \(\pi_P\) follows from the uniqueness of the critical
asymptotic part. We first identify the kernel of \(\pi_P\). By
Corollary~\ref{thm:global-decomposition-single-point}, an element
\(u\in D_{\max}(\Delta_X)\) has vanishing critical asymptotic part if
and only if \(u\in D_{\min}(\Delta_X)\). Therefore,
\(\ker\pi_P=D_{\min}(\Delta_X)\).
To prove surjectivity, let \(h\in V_P\), and choose
\(\chi\in C_c^\infty(D_P)\) equal to \(1\) near \(P\). Then
\(\chi h\in D_{\max}(\Delta_X)\), and its critical asymptotic part is
\(h\). Hence \(\pi_P(\chi h)=h\).
The asserted isomorphism now follows from the first isomorphism
theorem. Finally,
\(\dim_{\mathbb C}V_P=2(2n_P-1)=4n_P-2\).
\end{proof}

We now return to the case of an arbitrary finite conic set \(S\).

\begin{lem}\label{lem:global-remainder}
Let \(\chi_P\in C_c^\infty(D_P)\), \(P\in S\), be cutoff
functions equal to \(1\) near \(P\), with pairwise disjoint supports.
For \(u\in D_{\max}(\Delta_X)\), let \(h_{u,P}\in V_P\) be its
critical asymptotic part at \(P\). Then
\[
R
:=
u-\sum_{P\in S}\chi_P h_{u,P}
\in D_{\min}(\Delta_X).
\]
\end{lem}

\begin{proof}
Set \(\eta:=1-\sum_{P\in S}\chi_P\). Since the supports of the
cutoff functions are pairwise disjoint and each \(\chi_P\) is equal
to \(1\) near \(P\), the function \(\eta\) vanishes in a neighborhood
of \(S\). Since \(X\) is compact, it follows that
\(\operatorname{supp}\eta\Subset X'\), and hence
\(\operatorname{supp}(\eta u)\Subset X'\).
By interior elliptic regularity,
\(u\in H^2_{\mathrm{loc}}(X')\). Therefore,
\(\eta u\in D_{\max}(\Delta_X)\).
Lemma~\ref{lem:away-from-conic} then gives
\(\eta u\in D_{\min}(\Delta_X)\).
Moreover,
\[
R
=
u-\sum_{P\in S}\chi_P h_{u,P}
=
\sum_{P\in S}\chi_P(u-h_{u,P})+\eta u.
\]
By the local remainder statement,
\(\chi_P(u-h_{u,P})\in D_{\min}(\Delta_X)\) for every \(P\in S\).
Since \(S\) is finite, it follows that
\(R\in D_{\min}(\Delta_X)\).
\end{proof}

For each \(P\in S\), let \(V_P\) and \(\omega_P\) be the critical
asymptotic space and the normalized local Green form defined above.
Set \(V_S:=\bigoplus_{P\in S}V_P\). For
\(h=(h_P)_{P\in S}\) and \(\ell=(\ell_P)_{P\in S}\) in \(V_S\),
define
\(\omega_S(h,\ell):=\sum_{P\in S}\omega_P(h_P,\ell_P)\).
Then \(\omega_S\) is a nondegenerate skew-Hermitian sesquilinear form
on \(V_S\).

For each \(P\in S\), let
\(\pi_P:D_{\max}(\Delta_X)\to V_P\) be the map that assigns to
\(u\in D_{\max}(\Delta_X)\) its critical asymptotic part at \(P\).
Define the global critical asymptotic map
\(\pi_S:D_{\max}(\Delta_X)\to V_S\) by
\(\pi_S(u):=(\pi_P(u))_{P\in S}\).

\begin{lem}\label{lem:green-form-global}
For every \(u,v\in D_{\max}(\Delta_X)\), one has
\(q_X(u,v)=\omega_S\bigl(\pi_S(u),\pi_S(v)\bigr)\).
\end{lem}
\begin{proof}
Choose cutoff functions
\(\chi_P\in C_c^\infty(D_P)\), \(P\in S\), equal to \(1\) near \(P\),
with pairwise disjoint supports. By
Lemma~\ref{lem:global-remainder}, write
\[
u=\sum_{P\in S}\chi_P h_{u,P}+R_u,
\qquad
v=\sum_{P\in S}\chi_P h_{v,P}+R_v,
\]
where \(R_u,R_v\in D_{\min}(\Delta_X)\) and
\(h_{u,P}=\pi_P(u)\), \(h_{v,P}=\pi_P(v)\).
Lemma~\ref{lem:dmin-kernel} gives
\[
q_X(u,v)
=
q_X\left(
\sum_{P\in S}\chi_P h_{u,P},
\sum_{P\in S}\chi_P h_{v,P}
\right).
\]
Since the supports of the cutoff functions are pairwise disjoint, all
cross terms corresponding to distinct conic points vanish. Hence
\[
q_X(u,v)
=
\sum_{P\in S}
q_X(\chi_P h_{u,P},\chi_P h_{v,P}).
\]
By the local Green formula,
\(q_X(\chi_P h_{u,P},\chi_P h_{v,P})
=\omega_P(h_{u,P},h_{v,P})\).
Therefore,
\[
q_X(u,v)
=
\sum_{P\in S}\omega_P\bigl(\pi_P(u),\pi_P(v)\bigr)
=
\omega_S\bigl(\pi_S(u),\pi_S(v)\bigr).
\]
\end{proof}

\begin{cor}\label{cor:critical-asymptotic-projection}
For each \(P\in S\), let
\(\chi_P\in C_c^\infty(D_P)\) be equal to \(1\) near \(P\). Then, for
every \(u\in D_{\max}(\Delta_X)\),
\[
\pi_P(u)
=
\sum_{|j|\leq n_P-1}
\left(
q_X(u,\chi_Pf_{P,j}^{\#})f_{P,j}
-
q_X(u,\chi_Pf_{P,j})f_{P,j}^{\#}
\right).
\]
In particular, the right-hand side is independent of the choice of
\(\chi_P\).
\end{cor}

\begin{proof}
Write
\[
\pi_P(u)
=
\sum_{|j|\leq n_P-1}
\left(
a_jf_{P,j}+b_jf_{P,j}^{\#}
\right).
\]
By Lemma~\ref{lem:green-form-global} and the normalization of the
critical asymptotic terms, one has
\[
q_X(u,\chi_Pf_{P,j}^{\#})=a_j,
\qquad
q_X(u,\chi_Pf_{P,j})=-b_j.
\]
Substituting these identities into the expansion of \(\pi_P(u)\)
gives the asserted formula. Since \(\chi_P\) was arbitrary, the
right-hand side is independent of its choice.
\end{proof}

\begin{thm}\label{thm:quotient-general}
The map \(\pi_S\) is a surjective linear map with
\(\ker\pi_S=D_{\min}(\Delta_X)\). Consequently, it induces a linear
isomorphism
\[
\overline{\pi}_S:
D_{\max}(\Delta_X)/D_{\min}(\Delta_X)
\longrightarrow V_S,
\qquad
[u]\longmapsto\pi_S(u).
\]
In particular,
\(\dim_{\mathbb C}\bigl(D_{\max}(\Delta_X)/
D_{\min}(\Delta_X)\bigr)
=\sum_{P\in S}(4n_P-2)\).
\end{thm}
\begin{proof}
The linearity of \(\pi_S\) follows immediately from
Corollary~\ref{cor:critical-asymptotic-projection}.

We first prove surjectivity. Choose coordinate disks \(D_P\),
\(P\in S\), with pairwise disjoint closures, and cutoff functions
\(\chi_P\in C_c^\infty(D_P)\) equal to \(1\) near \(P\). Let
\(h=(h_P)_{P\in S}\in V_S\), and define
\(u:=\sum_{P\in S}\chi_P h_P\), where each \(\chi_P h_P\) is extended
by zero outside \(D_P\). Then \(u\in D_{\max}(\Delta_X)\) and
\(\pi_P(u)=h_P\) for every \(P\in S\). Hence \(\pi_S(u)=h\), proving
that \(\pi_S\) is surjective.

We next determine the kernel. Suppose first that
\(u\in\ker\pi_S\). Then \(\pi_P(u)=0\) for every \(P\in S\).
Lemma~\ref{lem:global-remainder} gives
\[
u
=
u-\sum_{P\in S}\chi_P\pi_P(u)
\in D_{\min}(\Delta_X).
\]
Thus \(\ker\pi_S\subset D_{\min}(\Delta_X)\).
Conversely, let \(u\in D_{\min}(\Delta_X)\). For every
\(v\in D_{\max}(\Delta_X)\), Lemma~\ref{lem:dmin-kernel} and
Lemma~\ref{lem:green-form-global} give
\[
0
=
q_X(u,v)
=
\omega_S\bigl(\pi_S(u),\pi_S(v)\bigr).
\]
Since \(\pi_S\) is surjective, \(\pi_S(v)\) ranges over all of \(V_S\).
The nondegeneracy of \(\omega_S\) therefore implies
\(\pi_S(u)=0\). Hence
\(u\in\ker\pi_S\), proving that
\(\ker\pi_S=D_{\min}(\Delta_X)\).

The asserted isomorphism follows from the first isomorphism theorem.
Finally,
\[
\dim_{\mathbb C}V_S
=
\sum_{P\in S}\dim_{\mathbb C}V_P
=
\sum_{P\in S}(4n_P-2).
\]
\end{proof}

\section{Self-Adjoint Extensions and the Weyl Function}
We use \((V_S,\omega_S)\) to parametrize the self-adjoint
realizations by Lagrangian subspaces and to construct their Poisson
operators and Weyl functions.

\begin{defn}\label{def:Lagrangian-Grassmannian}
For a subspace \(\mathcal V\subset V_S\), define its
\(\omega_S\)-orthogonal complement by
\[
\mathcal V^{\perp_{\omega_S}}
:=
\left\{
h\in V_S:
\omega_S(h,\ell)=0
\text{ for every }\ell\in\mathcal V
\right\}.
\]
The subspace \(\mathcal V\) is called isotropic if
\(\mathcal V\subset\mathcal V^{\perp_{\omega_S}}\), and Lagrangian if
\(\mathcal V=\mathcal V^{\perp_{\omega_S}}\). We denote by
\[
\operatorname{LGr}(V_S,\omega_S)
:=
\left\{
\mathcal V\subset V_S:
\mathcal V=\mathcal V^{\perp_{\omega_S}}
\right\}
\]
the Lagrangian Grassmannian associated with the nondegenerate
skew-Hermitian space \((V_S,\omega_S)\).
\end{defn}

We denote the minimal Laplacian by
\(
    \Delta_{X,\min}
    :=
    \Delta_X\big|_{D_{\min}(\Delta_X)}.
\)
For any subspace \(\mathcal V\subset V_S\), define
\[
    D_{\mathcal V}
    :=
    \left\{
        u\in D_{\max}(\Delta_X):
        \pi_S(u)\in\mathcal V
    \right\},
    \qquad
    \Delta_{X,\mathcal V}
    :=
    \Delta_X\big|_{D_{\mathcal V}}.
\]
Since \(\ker\pi_S=D_{\min}(\Delta_X)\), one has
\(D_{\{0\}}=D_{\min}(\Delta_X)\) and
\(\Delta_{X,\{0\}}=\Delta_{X,\min}\).

\begin{thm}\label{thm:self-adjoint-extensions}
For every subspace \(\mathcal V\subset V_S\), one has
\[
    \Delta_{X,\mathcal V}^{*}
    =
    \Delta_{X,\mathcal V^{\perp_{\omega_S}}}.
\]
In particular, \(\Delta_{X,\mathcal V}\) is self-adjoint if and only
if \(\mathcal V\) is Lagrangian. Moreover, the assignment
\(\mathcal V\mapsto\Delta_{X,\mathcal V}\) gives a bijection from
\(\operatorname{LGr}(V_S,\omega_S)\) onto the set of all
self-adjoint extensions of \(\Delta_{X,\min}\).
\end{thm}

\begin{proof}
Since \(\ker\pi_S=D_{\min}(\Delta_X)\), we have
\(D_{\min}(\Delta_X)\subset D_{\mathcal V}\). In particular,
\(C_c^\infty(X')\subset D_{\mathcal V}\), so
\(\Delta_{X,\mathcal V}\) is densely defined.

Let \(u\in D(\Delta_{X,\mathcal V}^{*})\). Testing the defining
identity for the adjoint against functions in \(C_c^\infty(X')\)
shows that \(u\in D_{\max}(\Delta_X)\) and
\(\Delta_{X,\mathcal V}^{*}u=\Delta_Xu\). Therefore,
\[
\begin{aligned}
u\in D(\Delta_{X,\mathcal V}^{*})
&\Longleftrightarrow
q_X(u,v)=0
\quad\text{for every }v\in D_{\mathcal V} \\
&\Longleftrightarrow
\omega_S\bigl(\pi_S(u),\ell\bigr)=0
\quad\text{for every }\ell\in\mathcal V \\
&\Longleftrightarrow
\pi_S(u)\in\mathcal V^{\perp_{\omega_S}}.
\end{aligned}
\]
Here we have used
\(\pi_S(D_{\mathcal V})=\mathcal V\), which follows from the
surjectivity of \(\pi_S\). Thus
\(D(\Delta_{X,\mathcal V}^{*})
=D_{\mathcal V^{\perp_{\omega_S}}}\). Since both operators act by
\(\Delta_X\), it follows that
\(\Delta_{X,\mathcal V}^{*}
=\Delta_{X,\mathcal V^{\perp_{\omega_S}}}\).

Consequently, \(\Delta_{X,\mathcal V}\) is self-adjoint if and only if
\(D_{\mathcal V}=D_{\mathcal V^{\perp_{\omega_S}}}\). Since
\(\pi_S\) is surjective, this is equivalent to
\(\mathcal V=\mathcal V^{\perp_{\omega_S}}\), namely, to
\(\mathcal V\) being Lagrangian.

Finally, let \(\widetilde{\Delta}\) be any self-adjoint extension of
the minimal conic Laplacian, and set
\(\mathcal D:=D(\widetilde{\Delta})\). Since
\(\widetilde{\Delta}\) is a self-adjoint extension of the minimal
Laplacian, one has
\(\mathcal D\subset D_{\max}(\Delta_X)\) and
\(\widetilde{\Delta}u=\Delta_Xu\) for \(u\in\mathcal D\).
Set \(\mathcal V:=\pi_S(\mathcal D)\).

We claim that \(\mathcal D=D_{\mathcal V}\). The inclusion
\(\mathcal D\subset D_{\mathcal V}\) is immediate. Conversely, if
\(u\in D_{\mathcal V}\), choose \(v\in\mathcal D\) such that
\(\pi_S(v)=\pi_S(u)\). Then
\[
    u-v\in\ker\pi_S
    =
    D_{\min}(\Delta_X)
    \subset\mathcal D,
\]
and hence \(u\in\mathcal D\). Therefore,
\(\mathcal D=D_{\mathcal V}\) and
\(\widetilde{\Delta}=\Delta_{X,\mathcal V}\).
Since \(\widetilde{\Delta}\) is self-adjoint, the preceding result
implies that \(\mathcal V\) is Lagrangian. The Lagrangian subspace is
unique because \(\mathcal V=\pi_S(\mathcal D)\), and hence the
correspondence is bijective.
\end{proof}

The elements of \(V_S\) will be called the critical asymptotic
boundary data. For a Lagrangian subspace
\(\mathcal V\subset V_S\), the condition
\(\pi_S(u)\in\mathcal V\) will be called the Lagrangian boundary
condition determined by \(\mathcal V\). The corresponding
self-adjoint realization of \(\Delta_X\), with domain
\(D_{\mathcal V}\), will be denoted by
\(\Delta_{X,\mathcal V}\).

Throughout the remainder of this paper, for
\(\mathcal V\in\operatorname{LGr}(V_S,\omega_S)\) and
\(\lambda\in\rho(\Delta_{X,\mathcal V})\), we denote the resolvent of
\(\Delta_{X,\mathcal V}\) by
\(
    R_{\mathcal V}(\lambda)
    :=
    (\lambda-\Delta_{X,\mathcal V})^{-1}.
\)

\begin{thm}\label{thm:compact-resolvent}
Let
\(\mathcal V\in\operatorname{LGr}(V_S,\omega_S)\), and let
\(\Delta_{X,\mathcal V}\) be the corresponding self-adjoint
realization. Then \(\Delta_{X,\mathcal V}\) is bounded from below and
has compact resolvent. Consequently, its spectrum is discrete and
consists of real eigenvalues of finite multiplicity that can
accumulate only at \(+\infty\).
\end{thm}

\begin{proof}
The compactness of the resolvent follows from the singular elliptic
estimate and the Rellich property for closed extensions of elliptic
cone operators. More precisely, the maximal domain embeds continuously
into a positively weighted cone Sobolev space, which is compactly
embedded in the underlying weighted \(L^2\)-space; see
\cite[Lemma~3.9 and (3.10)]{gjb}. Moreover,
\(\Delta_{X,\min}\) is nonnegative and has finite deficiency indices.
By the standard extension theory of semibounded symmetric operators
with finite deficiency indices, every self-adjoint extension of
\(\Delta_{X,\min}\) is bounded from below.
\end{proof}

Let \(\{(\lambda_i^{\mathcal V},\phi_i^{\mathcal V})\}_{i\geq0}\)
be the eigenpairs of \(\Delta_{X,\mathcal V}\), with the eigenvalues
listed in nondecreasing order and repeated according to multiplicity:
\(\lambda_0^{\mathcal V}\leq\lambda_1^{\mathcal V}\leq\cdots\) and
\(\lambda_i^{\mathcal V}\to+\infty\). We choose the eigenfunctions so
that \(\{\phi_i^{\mathcal V}\}_{i\geq0}\) is a complete orthonormal
basis of \(L^2(X')\). Thus
\(\Delta_{X,\mathcal V}\phi_i^{\mathcal V}
=\lambda_i^{\mathcal V}\phi_i^{\mathcal V}\).
For each \(i\geq0\), define the linear functional
\(q_i^{\mathcal V}:D_{\max}(\Delta_X)\to\mb C\) by
\(
    q_i^{\mathcal V}(u)
    :=
    q_X(u,\phi_i^{\mathcal V}).
\)
\begin{lem}\label{lem:Y-series}
Let
\(\mathcal V\in\operatorname{LGr}(V_S,\omega_S)\) and
\(\lambda\in\rho(\Delta_{X,\mathcal V})\). For every
\(u\in D_{\max}(\Delta_X)\), the series
\(
    \sum_{i\geq0}
    (\lambda-\lambda_i^{\mathcal V})^{-1}
    q_i^{\mathcal V}(u)\phi_i^{\mathcal V}
\)
converges in \(L^2(X')\). Define
\[
    Y_{\mathcal V}(\lambda):
    D_{\max}(\Delta_X)
    \longrightarrow
    L^2(X')
\]
by
\(
    Y_{\mathcal V}(\lambda)u
    :=
    \sum_{i\geq0}
    (\lambda-\lambda_i^{\mathcal V})^{-1}
    q_i^{\mathcal V}(u)\phi_i^{\mathcal V}.
\)
Then
\(
    Y_{\mathcal V}(\lambda)u
    =
    u+R_{\mathcal V}(\lambda)(\Delta_X-\lambda)u.
\)
\end{lem}

\begin{proof}
For brevity, write
\(\langle\cdot,\cdot\rangle\) for the inner product on \(L^2(X')\),
and set \(f:=(\Delta_X-\lambda)u\). Since
\(u\in D_{\max}(\Delta_X)\), one has \(f\in L^2(X')\). For each
\(i\geq0\),
\[
    q_i^{\mathcal V}(u)
    =
    \langle f,\phi_i^{\mathcal V}\rangle
    +
    (\lambda-\lambda_i^{\mathcal V})
    \langle u,\phi_i^{\mathcal V}\rangle.
\]
Consequently,
\[
    \frac{q_i^{\mathcal V}(u)}
    {\lambda-\lambda_i^{\mathcal V}}
    =
    \langle u,\phi_i^{\mathcal V}\rangle
    +
    \frac{\langle f,\phi_i^{\mathcal V}\rangle}
    {\lambda-\lambda_i^{\mathcal V}}.
\]
The first term on the right is the \(i\)-th Fourier coefficient of
\(u\) with respect to the orthonormal basis
\(\{\phi_i^{\mathcal V}\}_{i\geq0}\), while the second is the
corresponding coefficient in the spectral expansion of
\(R_{\mathcal V}(\lambda)f\). Both series converge in \(L^2(X')\).
Hence
\(
    Y_{\mathcal V}(\lambda)u
    =
    u+R_{\mathcal V}(\lambda)f
    =
    u+R_{\mathcal V}(\lambda)(\Delta_X-\lambda)u.
\)
\end{proof}

The identity in Lemma~\ref{lem:Y-series} shows that
\(Y_{\mathcal V}(\lambda)u\in D_{\max}(\Delta_X)\) for every
\(u\in D_{\max}(\Delta_X)\). Thus we regard
\(Y_{\mathcal V}(\lambda)\) as a linear operator on
\(D_{\max}(\Delta_X)\).

\begin{cor}\label{cor:Y-projection}
Let
\(\mathcal V\in\operatorname{LGr}(V_S,\omega_S)\) and
\(\lambda\in\rho(\Delta_{X,\mathcal V})\). Then
\(Y_{\mathcal V}(\lambda)\) is the projection of
\(D_{\max}(\Delta_X)\) onto \(\ker(\Delta_X-\lambda)\) along
\(D_{\mathcal V}\). Consequently,
\[
D_{\max}(\Delta_X)
=
D_{\mathcal V}\oplus\ker(\Delta_X-\lambda).
\]
Moreover,
\(
\operatorname{rank}Y_{\mathcal V}(\lambda)
=
\frac12\dim_{\mathbb C}V_S.
\)
\end{cor}

Indeed, the identity in Lemma~\ref{lem:Y-series} shows that
\(Y_{\mathcal V}(\lambda)\) is the identity on
\(\ker(\Delta_X-\lambda)\) and that
\(\ker Y_{\mathcal V}(\lambda)=D_{\mathcal V}\). The rank formula
follows from
\(
D_{\max}(\Delta_X)/D_{\mathcal V}\cong V_S/\mathcal V
\)
and
\(
\dim_{\mathbb C}\mathcal V
=
\frac12\dim_{\mathbb C}V_S.
\)

\begin{defn}
A Lagrangian pair in \((V_S,\omega_S)\) is an ordered pair
\((\mathcal V,\mathcal V^\#)\) of Lagrangian subspaces satisfying
\(V_S=\mathcal V\oplus\mathcal V^\#\).
\end{defn}

Let \((\mathcal V,\mathcal V^\#)\) be a Lagrangian pair in
\((V_S,\omega_S)\). Denote by
\(P_{\mathcal V}:V_S\to\mathcal V\) and
\(P_{\mathcal V^\#}:V_S\to\mathcal V^\#\) the corresponding
projections. Define
\(\widetilde{\pi}_{\mathcal V}
:=P_{\mathcal V}\circ\pi_S:
D_{\max}(\Delta_X)\to\mathcal V\)
and
\(\widetilde{\pi}_{\mathcal V^\#}
:=P_{\mathcal V^\#}\circ\pi_S:
D_{\max}(\Delta_X)\to\mathcal V^\#\).
Thus
\(
\pi_S(u)
=
\widetilde{\pi}_{\mathcal V}(u)
+
\widetilde{\pi}_{\mathcal V^\#}(u)
\)
for every \(u\in D_{\max}(\Delta_X)\).

Fix \(\lambda\in\rho(\Delta_{X,\mathcal V})\). By
Corollary~\ref{cor:Y-projection},
\(Y_{\mathcal V}(\lambda)\) is a projection satisfying
\[
\ker Y_{\mathcal V}(\lambda)=D_{\mathcal V},
\qquad
\operatorname{ran}Y_{\mathcal V}(\lambda)
=
\ker(\Delta_X-\lambda).
\]
Therefore, it induces a linear isomorphism
\[
\overline{Y}_{\mathcal V}(\lambda):
D_{\max}(\Delta_X)/D_{\mathcal V}
\longrightarrow
\ker(\Delta_X-\lambda),
\qquad
[u]\longmapsto Y_{\mathcal V}(\lambda)u.
\]

Moreover, by definition,
\(\ker\widetilde{\pi}_{\mathcal V^\#}=D_{\mathcal V}\).
Since \(\pi_S\) is surjective, so is
\(\widetilde{\pi}_{\mathcal V^\#}\). Hence, by the first
isomorphism theorem, it induces a linear isomorphism
\[
\overline{\pi}_{\mathcal V^\#}:
D_{\max}(\Delta_X)/D_{\mathcal V}
\longrightarrow
\mathcal V^\#,
\qquad
[u]\longmapsto
\widetilde{\pi}_{\mathcal V^\#}(u).
\]
We define the Poisson operator
\[
\Gamma_{\mathcal V,\mathcal V^\#}^{X}(\lambda):
\mathcal V^\#
\longrightarrow
\ker(\Delta_X-\lambda)
\]
by
\(
\Gamma_{\mathcal V,\mathcal V^\#}^{X}(\lambda)
:=
\overline{Y}_{\mathcal V}(\lambda)
\circ
\overline{\pi}_{\mathcal V^\#}^{-1}.
\)
Equivalently, if \(f\in\mathcal V^\#\) and
\(u\in D_{\max}(\Delta_X)\) satisfies
\(\widetilde{\pi}_{\mathcal V^\#}(u)=f\), then
\(
\Gamma_{\mathcal V,\mathcal V^\#}^{X}(\lambda)f
=
Y_{\mathcal V}(\lambda)u.
\)
Consequently,
\(
Y_{\mathcal V}(\lambda)
=
\Gamma_{\mathcal V,\mathcal V^\#}^{X}(\lambda)
\circ
\widetilde{\pi}_{\mathcal V^\#}.
\)

\begin{defn}\label{defn:formal-solution}
Let \(\lambda\in\mb C\). A function
\(u_\lambda\in D_{\max}(\Delta_X)\) is called a formal solution of
\(\Delta_Xu_\lambda=\lambda u_\lambda\) if
\((\Delta_X-\lambda)u_\lambda=0\) in the sense of distributions on
\(X'\). We denote the space of formal solutions at \(\lambda\) by
\(\ker(\Delta_X-\lambda)\). Throughout this section, this kernel is
understood to be taken in \(D_{\max}(\Delta_X)\).
\end{defn}

\begin{prop}\label{prop:Gamma-characterization}
Let \(\lambda\in\rho(\Delta_{X,\mathcal V})\) and
\(f\in\mathcal V^\#\). A function
\(u_\lambda\in D_{\max}(\Delta_X)\) is a formal solution satisfying
the boundary condition
\(\widetilde{\pi}_{\mathcal V^\#}(u_\lambda)=f\) if and only if
\(u_\lambda
=\Gamma_{\mathcal V,\mathcal V^\#}^{X}(\lambda)f\).
\end{prop}

\begin{proof}
Suppose that \(u_\lambda\) is a formal solution satisfying the
boundary condition, and set
\(v_\lambda:=
\Gamma_{\mathcal V,\mathcal V^\#}^{X}(\lambda)f\).
By definition,
\(v_\lambda\in\ker(\Delta_X-\lambda)\) and
\(\widetilde{\pi}_{\mathcal V^\#}(v_\lambda)=f\). Hence
\(\widetilde{\pi}_{\mathcal V^\#}
(u_\lambda-v_\lambda)=0\), so
\(u_\lambda-v_\lambda\in D_{\mathcal V}\).
Since both \(u_\lambda\) and \(v_\lambda\) are formal solutions, we
have
\((\Delta_{X,\mathcal V}-\lambda)
(u_\lambda-v_\lambda)=0\).
Since \(\lambda\in\rho(\Delta_{X,\mathcal V})\), it follows that
\(u_\lambda-v_\lambda=0\). Hence
\(u_\lambda=v_\lambda
=\Gamma_{\mathcal V,\mathcal V^\#}^{X}(\lambda)f\).
The converse follows immediately from the definition of
\(\Gamma_{\mathcal V,\mathcal V^\#}^{X}(\lambda)\).
\end{proof}

Thus \(\Gamma_{\mathcal V,\mathcal V^\#}^{X}(\lambda)\) parametrizes
formal solutions uniquely by their \(\mathcal V^\#\)-boundary
coordinates.

\begin{defn}\label{defn:Weyl-function}
The Weyl function associated with the Lagrangian pair
\((\mathcal V,\mathcal V^\#)\) is the operator-valued function
\(M_{\mathcal V,\mathcal V^\#}^{X}:
\rho(\Delta_{X,\mathcal V})\to
\operatorname{Hom}_{\mb C}(\mathcal V^\#,\mathcal V)\)
defined by
\[
M_{\mathcal V,\mathcal V^\#}^{X}(\lambda)
:=
\widetilde{\pi}_{\mathcal V}\circ
\Gamma_{\mathcal V,\mathcal V^\#}^{X}(\lambda).
\]
\end{defn}

Let \(\lambda\in\rho(\Delta_{X,\mathcal V})\) and
\(f\in\mathcal V^\#\). Since
\(\widetilde{\pi}_{\mathcal V^\#}
\bigl(
\Gamma_{\mathcal V,\mathcal V^\#}^{X}(\lambda)f
\bigr)=f\), the decomposition
\(V_S=\mathcal V\oplus\mathcal V^\#\) gives
\(
\pi_S\bigl(
\Gamma_{\mathcal V,\mathcal V^\#}^{X}(\lambda)f
\bigr)
=
M_{\mathcal V,\mathcal V^\#}^{X}(\lambda)f+f.
\)
Equivalently, if \(u\in D_{\max}(\Delta_X)\) satisfies
\(\widetilde{\pi}_{\mathcal V^\#}(u)=f\), then
\(Y_{\mathcal V}(\lambda)u
=\Gamma_{\mathcal V,\mathcal V^\#}^{X}(\lambda)f\), and hence
\(
\pi_S\bigl(Y_{\mathcal V}(\lambda)u\bigr)
=
M_{\mathcal V,\mathcal V^\#}^{X}(\lambda)f+f.
\)

\begin{defn}\label{defn:symplectic-adjoint}
Let \((\mathcal V,\mathcal V^\#)\) be a Lagrangian pair in
\((V_S,\omega_S)\). For a linear map
\(A:\mathcal V^\#\to\mathcal V\), its symplectic adjoint is the
unique linear map
\(A^{\dagger_{\omega_S}}:\mathcal V^\#\to\mathcal V\) satisfying
\(
\omega_S(Af,g)
=
\omega_S(f,A^{\dagger_{\omega_S}}g)
\)
for all \(f,g\in\mathcal V^\#\). Equivalently,
\(
\omega_S(Af,g)
=
-\overline{
\omega_S(A^{\dagger_{\omega_S}}g,f)
}.
\)
Existence and uniqueness follow from the nondegeneracy of the pairing
between \(\mathcal V\) and \(\mathcal V^\#\) induced by
\(\omega_S\). When no confusion can arise, we write
\(A^\dagger\) for \(A^{\dagger_{\omega_S}}\).
\end{defn}

\begin{lem}\label{lem:weyl-function-hermitian-symmetry}
Let \((\mathcal V,\mathcal V^\#)\) be a Lagrangian pair in
\((V_S,\omega_S)\). Then, for every
\(\lambda\in\rho(\Delta_{X,\mathcal V})\), one has
\[
M_{\mathcal V,\mathcal V^\#}^{X}(\lambda)
^{\dagger_{\omega_S}}
=
-M_{\mathcal V,\mathcal V^\#}^{X}(\overline{\lambda}).
\]
Consequently, if \(\beta\) is a basis of \(\mathcal V\) and
\(\beta^\#\) is its \(\omega_S\)-dual basis in \(\mathcal V^\#\),
then
\[
[M_{\mathcal V,\mathcal V^\#}^{X}(\lambda)]_{\beta^\#}^{\beta}
=
([M_{\mathcal V,\mathcal V^\#}^{X}(\overline{\lambda})]
_{\beta^\#}^{\beta})^*.
\]
In particular, if
\(\lambda\in\rho(\Delta_{X,\mathcal V})\cap\mb R\), then
\([M_{\mathcal V,\mathcal V^\#}^{X}(\lambda)]
_{\beta^\#}^{\beta}\) is Hermitian.
\end{lem}

\begin{proof}
Since \(\Delta_{X,\mathcal V}\) is self-adjoint,
\(\overline{\lambda}\in\rho(\Delta_{X,\mathcal V})\). Let
\(f,g\in\mathcal V^\#\), and set
\(u_\lambda:=
\Gamma_{\mathcal V,\mathcal V^\#}^{X}(\lambda)f\) and
\(v_{\overline{\lambda}}:=
\Gamma_{\mathcal V,\mathcal V^\#}^{X}
(\overline{\lambda})g\).
Since
\(\Delta_Xu_\lambda=\lambda u_\lambda\) and
\(\Delta_Xv_{\overline{\lambda}}
=\overline{\lambda}v_{\overline{\lambda}}\), one has
\(q_X(u_\lambda,v_{\overline{\lambda}})=0\).

The global Green formula, the definition of the Weyl function, and
the fact that both \(\mathcal V\) and \(\mathcal V^\#\) are
Lagrangian give
\[
\omega_S\bigl(
M_{\mathcal V,\mathcal V^\#}^{X}(\lambda)f,g
\bigr)
=
-\omega_S\bigl(
f,
M_{\mathcal V,\mathcal V^\#}^{X}(\overline{\lambda})g
\bigr).
\]
Equivalently,
\[
\omega_S\bigl(
M_{\mathcal V,\mathcal V^\#}^{X}(\lambda)f,g
\bigr)
=
\omega_S\bigl(
f,
-M_{\mathcal V,\mathcal V^\#}^{X}(\overline{\lambda})g
\bigr).
\]
By the definition of the symplectic adjoint,
\(
M_{\mathcal V,\mathcal V^\#}^{X}(\lambda)
^{\dagger_{\omega_S}}
=
-M_{\mathcal V,\mathcal V^\#}^{X}(\overline{\lambda}).
\)

With respect to the \(\omega_S\)-dual bases \(\beta^\#\) and
\(\beta\), one has
\(
[A^{\dagger_{\omega_S}}]_{\beta^\#}^{\beta}
=
-([A]_{\beta^\#}^{\beta})^*.
\)
Applying this identity to
\(A=M_{\mathcal V,\mathcal V^\#}^{X}(\lambda)\) gives
\(
[M_{\mathcal V,\mathcal V^\#}^{X}(\overline{\lambda})]
_{\beta^\#}^{\beta}
=
([M_{\mathcal V,\mathcal V^\#}^{X}(\lambda)]
_{\beta^\#}^{\beta})^*,
\)
which is equivalent to the asserted identity. If \(\lambda\) is
real, it follows that
\([M_{\mathcal V,\mathcal V^\#}^{X}(\lambda)]
_{\beta^\#}^{\beta}\) is Hermitian.
\end{proof}

For each \(i\geq0\), recall that
\(q_i^{\mathcal V}:D_{\max}(\Delta_X)\to\mb C\) is defined by
\(q_i^{\mathcal V}(u):=q_X(u,\phi_i^{\mathcal V})\).
Since \(\phi_i^{\mathcal V}\in D_{\mathcal V}\) and
\(\Delta_{X,\mathcal V}\) is self-adjoint, one has
\(q_i^{\mathcal V}(u)=0\) for every \(u\in D_{\mathcal V}\).
Thus \(q_i^{\mathcal V}\) induces a linear functional
\(\overline{q}_i^{\mathcal V}:
D_{\max}(\Delta_X)/D_{\mathcal V}\to\mb C\), given by
\(\overline{q}_i^{\mathcal V}([u])=q_i^{\mathcal V}(u)\).

Define
\(\omega_i^{\mathcal V,\mathcal V^\#}:
\mathcal V^\#\to\mb C\) by
\(\omega_i^{\mathcal V,\mathcal V^\#}
:=\overline{q}_i^{\mathcal V}
\circ\overline{\pi}_{\mathcal V^\#}^{-1}\).
Equivalently, if \(u\in D_{\max}(\Delta_X)\) satisfies
\(\widetilde{\pi}_{\mathcal V^\#}(u)=f\), then
\(\omega_i^{\mathcal V,\mathcal V^\#}(f)
=q_i^{\mathcal V}(u)\).

It follows from the eigenfunction expansion of
\(Y_{\mathcal V}(\lambda)\) and Parseval's identity that, for
\(\lambda,\mu\in\rho(\Delta_{X,\mathcal V})\) and
\(f,g\in\mathcal V^\#\),
\[
\left\langle
\Gamma_{\mathcal V,\mathcal V^\#}^{X}(\lambda)f,
\Gamma_{\mathcal V,\mathcal V^\#}^{X}(\mu)g
\right\rangle_{L^2(X')}
=
\sum_{i\geq0}
\frac{
\omega_i^{\mathcal V,\mathcal V^\#}(f)
\overline{\omega_i^{\mathcal V,\mathcal V^\#}(g)}
}{
(\lambda-\lambda_i^{\mathcal V})
(\overline{\mu}-\lambda_i^{\mathcal V})
}.
\]

Since \(\Delta_{X,\mathcal V}\) is self-adjoint,
\(\lambda\in\rho(\Delta_{X,\mathcal V})\) implies
\(\overline{\lambda}\in\rho(\Delta_{X,\mathcal V})\). Taking
\(\mu=\overline{\lambda}\), we obtain
\[
\left\langle
\Gamma_{\mathcal V,\mathcal V^\#}^{X}(\lambda)f,
\Gamma_{\mathcal V,\mathcal V^\#}^{X}(\overline{\lambda})g
\right\rangle_{L^2(X')}
=
\sum_{i\geq0}
\frac{
\omega_i^{\mathcal V,\mathcal V^\#}(f)
\overline{\omega_i^{\mathcal V,\mathcal V^\#}(g)}
}{
(\lambda-\lambda_i^{\mathcal V})^2
}.
\]

\begin{lem}\label{lem:M-derivative-identity}
The Weyl function
\(M_{\mathcal V,\mathcal V^\#}^{X}\) is holomorphic on
\(\rho(\Delta_{X,\mathcal V})\). For
\(f,g\in\mathcal V^\#\) and
\(\lambda\in\rho(\Delta_{X,\mathcal V})\), one has
\[
\omega_S\left(
\frac{d}{d\lambda}
M_{\mathcal V,\mathcal V^\#}^{X}(\lambda)f,g
\right)
=
\left\langle
\Gamma_{\mathcal V,\mathcal V^\#}^{X}(\lambda)f,
\Gamma_{\mathcal V,\mathcal V^\#}^{X}
(\overline{\lambda})g
\right\rangle_{L^2(X')}.
\]
Equivalently,
\[
\omega_S\left(
\frac{d}{d\lambda}
M_{\mathcal V,\mathcal V^\#}^{X}(\lambda)f,g
\right)
=
\sum_{i\geq0}
\frac{
\omega_i^{\mathcal V,\mathcal V^\#}(f)
\overline{\omega_i^{\mathcal V,\mathcal V^\#}(g)}
}{
(\lambda-\lambda_i^{\mathcal V})^2
},
\]
where
\(\omega_i^{\mathcal V,\mathcal V^\#}
:=
\overline{q}_i^{\mathcal V}
\circ\overline{\pi}_{\mathcal V^\#}^{-1}\).
\end{lem}

\begin{proof}
Fix \(f\in\mathcal V^\#\), and choose
\(w_f\in D_{\max}(\Delta_X)\) such that
\(\widetilde{\pi}_{\mathcal V^\#}(w_f)=f\). Then
\[
\Gamma_{\mathcal V,\mathcal V^\#}^{X}(\lambda)f
=
Y_{\mathcal V}(\lambda)w_f
=
w_f+
R_{\mathcal V}(\lambda)(\Delta_X-\lambda)w_f.
\]
The resolvent \(R_{\mathcal V}(\lambda)\) is holomorphic on
\(\rho(\Delta_{X,\mathcal V})\), with values in
\(D_{\mathcal V}\) equipped with the graph norm. Hence
\(\Gamma_{\mathcal V,\mathcal V^\#}^{X}(\lambda)f\) is
holomorphic as a \(D_{\max}(\Delta_X)\)-valued function.
Since \(\widetilde{\pi}_{\mathcal V}\) is continuous with respect
to the graph norm, it follows that
\(M_{\mathcal V,\mathcal V^\#}^{X}\) is holomorphic.
Set
\(u_f(\lambda):=
\Gamma_{\mathcal V,\mathcal V^\#}^{X}(\lambda)f\) and
\(u_g(\overline{\lambda}):=
\Gamma_{\mathcal V,\mathcal V^\#}^{X}
(\overline{\lambda})g\).
Then
\((\Delta_X-\lambda)u_f(\lambda)=0\) and
\((\Delta_X-\overline{\lambda})
u_g(\overline{\lambda})=0\).
Moreover,
\(\pi_S(u_f(\lambda))
=M_{\mathcal V,\mathcal V^\#}^{X}(\lambda)f+f\).
Differentiating with respect to \(\lambda\), we obtain
\(
\pi_S\bigl(u_f'(\lambda)\bigr)
=
\frac{d}{d\lambda}
M_{\mathcal V,\mathcal V^\#}^{X}(\lambda)f
\)
and
\(
(\Delta_X-\lambda)u_f'(\lambda)=u_f(\lambda).
\)
Using the definition of \(q_X\) and the convention that the
\(L^2(X')\)-inner product is linear in the first variable, we obtain
\[
\begin{aligned}
q_X\bigl(u_f'(\lambda),u_g(\overline{\lambda})\bigr)
&=
\left\langle
\lambda u_f'(\lambda)+u_f(\lambda),
u_g(\overline{\lambda})
\right\rangle_{L^2(X')} \\
&\quad-
\left\langle
u_f'(\lambda),
\overline{\lambda}u_g(\overline{\lambda})
\right\rangle_{L^2(X')} \\
&=
\left\langle
u_f(\lambda),u_g(\overline{\lambda})
\right\rangle_{L^2(X')}.
\end{aligned}
\]
On the other hand, the global Green formula gives
\[
\begin{aligned}
q_X\bigl(u_f'(\lambda),u_g(\overline{\lambda})\bigr)
&=
\omega_S\left(
\frac{d}{d\lambda}
M_{\mathcal V,\mathcal V^\#}^{X}(\lambda)f,
M_{\mathcal V,\mathcal V^\#}^{X}
(\overline{\lambda})g+g
\right) \\
&=
\omega_S\left(
\frac{d}{d\lambda}
M_{\mathcal V,\mathcal V^\#}^{X}(\lambda)f,g
\right),
\end{aligned}
\]
because both
\(\frac{d}{d\lambda}
M_{\mathcal V,\mathcal V^\#}^{X}(\lambda)f\)
and
\(M_{\mathcal V,\mathcal V^\#}^{X}
(\overline{\lambda})g\)
belong to the Lagrangian subspace \(\mathcal V\).
Comparing the two expressions proves the first identity. The series
representation follows from the Parseval identity established above.
\end{proof}

\begin{cor}\label{cor:weyl-function-simple-poles}
Let \((\mathcal V,\mathcal V^\#)\) be a Lagrangian pair in
\((V_S,\omega_S)\). Then the associated Weyl function
\(M_{\mathcal V,\mathcal V^\#}^{X}(\lambda)\) extends
meromorphically across the eigenvalues of
\(\Delta_{X,\mathcal V}\), and all its possible poles are simple.

More precisely, let \(\lambda_0\) be an eigenvalue of
\(\Delta_{X,\mathcal V}\). Then, near \(\lambda_0\),
\[
M_{\mathcal V,\mathcal V^\#}^{X}(\lambda)
=
-\frac{B_{\lambda_0}}{\lambda-\lambda_0}
+
H_{\lambda_0}(\lambda),
\]
where \(H_{\lambda_0}\) is holomorphic near \(\lambda_0\), and
\(B_{\lambda_0}:\mathcal V^\#\to\mathcal V\) is uniquely determined
by
\[
\omega_S(B_{\lambda_0}f,g)
=
\sum_{\lambda_i^{\mathcal V}=\lambda_0}
\omega_i^{\mathcal V,\mathcal V^\#}(f)
\overline{\omega_i^{\mathcal V,\mathcal V^\#}(g)}
\]
for all \(f,g\in\mathcal V^\#\). In particular, \(\lambda_0\) is an
actual pole of \(M_{\mathcal V,\mathcal V^\#}^{X}\) if and only if
\(B_{\lambda_0}\neq0\).
\end{cor}

\begin{proof}
Let \(\lambda_0\) be an eigenvalue of
\(\Delta_{X,\mathcal V}\). Since \(\Delta_{X,\mathcal V}\) has
compact resolvent, \(\lambda_0\) is isolated and has finite
multiplicity. By Lemma~\ref{lem:M-derivative-identity}, the singular
part at \(\lambda_0\) of
\[
\omega_S\left(
\frac{d}{d\lambda}
M_{\mathcal V,\mathcal V^\#}^{X}(\lambda)f,g
\right)
\]
is
\[
\frac{1}{(\lambda-\lambda_0)^2}
\sum_{\lambda_i^{\mathcal V}=\lambda_0}
\omega_i^{\mathcal V,\mathcal V^\#}(f)
\overline{\omega_i^{\mathcal V,\mathcal V^\#}(g)}.
\]

The remaining part of the spectral series converges locally
uniformly near \(\lambda_0\). Indeed, fix
\(\zeta\in\rho(\Delta_{X,\mathcal V})\) near \(\lambda_0\). The
eigenfunction expansion of the Poisson operator shows that
\[
\left\{
\frac{\omega_i^{\mathcal V,\mathcal V^\#}(f)}
{\zeta-\lambda_i^{\mathcal V}}
\right\}_{i\geq0},
\qquad
\left\{
\frac{\omega_i^{\mathcal V,\mathcal V^\#}(g)}
{\zeta-\lambda_i^{\mathcal V}}
\right\}_{i\geq0}
\]
belong to \(\ell^2\). The Cauchy--Schwarz inequality therefore gives
the required locally uniform convergence after the terms satisfying
\(\lambda_i^{\mathcal V}=\lambda_0\) are removed. Hence the
remaining part is holomorphic near \(\lambda_0\).

The pairing
\(\omega_S|_{\mathcal V\times\mathcal V^\#}\) is nondegenerate.
Therefore, there exists a unique linear map
\(B_{\lambda_0}:\mathcal V^\#\to\mathcal V\) satisfying the identity
in the statement. It follows that
\[
\frac{d}{d\lambda}
M_{\mathcal V,\mathcal V^\#}^{X}(\lambda)
=
\frac{B_{\lambda_0}}{(\lambda-\lambda_0)^2}
+
K_{\lambda_0}(\lambda),
\]
where \(K_{\lambda_0}\) is holomorphic near \(\lambda_0\).
Consequently, the derivative of
\(
M_{\mathcal V,\mathcal V^\#}^{X}(\lambda)
+B_{\lambda_0}/(\lambda-\lambda_0)
\)
extends holomorphically across \(\lambda_0\). It follows that
\[
M_{\mathcal V,\mathcal V^\#}^{X}(\lambda)
=
-\frac{B_{\lambda_0}}{\lambda-\lambda_0}
+
H_{\lambda_0}(\lambda)
\]
for some \(H_{\lambda_0}\) holomorphic near \(\lambda_0\).
Therefore, every possible pole is simple, and \(\lambda_0\) is an
actual pole precisely when \(B_{\lambda_0}\neq0\).
\end{proof}

\section{Kre\u{\i}n's Formula}

We derive Kreĭn's formula for the resolvent difference of two
self-adjoint realizations and the resulting trace identity. Fix a
Lagrangian pair
\((\mathcal W,\mathcal W^\#)\), and denote by
\(P_{\mathcal W}:V_S\to\mathcal W\) and
\(P_{\mathcal W^\#}:V_S\to\mathcal W^\#\) the corresponding
projections. We take \(\Delta_{X,\mathcal W}\) as the reference
realization and express the formulas through the Weyl function of
\((\mathcal W,\mathcal W^\#)\).

Let \((\mathcal V,\mathcal V^\#)\) be another Lagrangian pair.
Define
\(A_{\mathcal V}
:=P_{\mathcal W}|_{\mathcal V}:
\mathcal V\to\mathcal W\)
and
\(C_{\mathcal V}
:=P_{\mathcal W^\#}|_{\mathcal V}:
\mathcal V\to\mathcal W^\#\).
Thus every \(v\in\mathcal V\) decomposes as
\(v=A_{\mathcal V}v+C_{\mathcal V}v\).

Define the symplectic adjoints
\(C_{\mathcal V}^{\dagger_{\omega_S}}:
\mathcal W\to\mathcal V^\#\)
and
\(A_{\mathcal V}^{\dagger_{\omega_S}}:
\mathcal W^\#\to\mathcal V^\#\)
by
\(\omega_S(v,C_{\mathcal V}^{\dagger_{\omega_S}}g)
=\omega_S(C_{\mathcal V}v,g)\)
and
\(\omega_S(v,A_{\mathcal V}^{\dagger_{\omega_S}}h)
=\omega_S(A_{\mathcal V}v,h)\)
for \(v\in\mathcal V\), \(g\in\mathcal W\), and
\(h\in\mathcal W^\#\). These adjoints exist uniquely by the
nondegeneracy of the pairings between complementary Lagrangian
subspaces.

Let \(g\in\mathcal W\) and \(h\in\mathcal W^\#\). Since
\(\mathcal V\) is Lagrangian, \(g+h\in\mathcal V\) if and only if
\(\omega_S(g+h,v)=0\) for every \(v\in\mathcal V\). Writing
\(v=A_{\mathcal V}v+C_{\mathcal V}v\), we obtain
\(
\omega_S(g+h,v)
=
\omega_S(
C_{\mathcal V}^{\dagger_{\omega_S}}g
+
A_{\mathcal V}^{\dagger_{\omega_S}}h,
v).
\)
By the nondegeneracy of
\(\omega_S|_{\mathcal V^\#\times\mathcal V}\), it follows that
\(g+h\in\mathcal V\) if and only if
\(C_{\mathcal V}^{\dagger_{\omega_S}}g
+A_{\mathcal V}^{\dagger_{\omega_S}}h=0\). Consequently,
\(
\mathcal V
=
\ker(
C_{\mathcal V}^{\dagger_{\omega_S}}P_{\mathcal W}
+
A_{\mathcal V}^{\dagger_{\omega_S}}P_{\mathcal W^\#}).
\)

For \(\lambda\in\rho(\Delta_{X,\mathcal W})\), define the linear map
\[
T_{(\mathcal V,\mathcal V^\#),(\mathcal W,\mathcal W^\#)}
:\mathcal W^\#\longrightarrow\mathcal V^\#
\]
by
\[
T_{(\mathcal V,\mathcal V^\#),(\mathcal W,\mathcal W^\#)}(\lambda)
=
C_{\mathcal V}^{\dagger_{\omega_S}}
M_{\mathcal W,\mathcal W^\#}^{X}(\lambda)
+
A_{\mathcal V}^{\dagger_{\omega_S}}.
\]
Once both Lagrangian pairs have been fixed, we abbreviate this
operator as \(T_{\mathcal V,\mathcal W}(\lambda)\).

\begin{lem}\label{lem:boundary-operator-invertible}
If
\(\lambda\in
\rho(\Delta_{X,\mathcal V})
\cap\rho(\Delta_{X,\mathcal W})\), then
\(T_{\mathcal V,\mathcal W}(\lambda)\) is a linear isomorphism from
\(\mathcal W^\#\) onto \(\mathcal V^\#\).
\end{lem}

\begin{proof}
Since
\(\dim_{\mb C}\mathcal V^\#
=\dim_{\mb C}\mathcal W^\#\), it suffices to prove injectivity.
Let \(f\in\ker T_{\mathcal V,\mathcal W}(\lambda)\). Then
\(
C_{\mathcal V}^{\dagger_{\omega_S}}
M_{\mathcal W,\mathcal W^\#}^{X}(\lambda)f
+
A_{\mathcal V}^{\dagger_{\omega_S}}f
=0.
\)
By the boundary equation for \(\mathcal V\), this implies
\(
M_{\mathcal W,\mathcal W^\#}^{X}(\lambda)f+f
\in\mathcal V.
\)

Let
\(u_f(\lambda):=
\Gamma_{\mathcal W,\mathcal W^\#}^{X}(\lambda)f\).
Then \(u_f(\lambda)\) is the formal solution at \(\lambda\) whose
\(\mathcal W^\#\)-boundary coordinate is \(f\). By the definition of
the Weyl function,
\(
\pi_S(u_f(\lambda))
=
M_{\mathcal W,\mathcal W^\#}^{X}(\lambda)f+f.
\)
Hence \(\pi_S(u_f(\lambda))\in\mathcal V\), and therefore
\(u_f(\lambda)\in D_{\mathcal V}\). Since
\((\Delta_X-\lambda)u_f(\lambda)=0\), we have
\(
(\Delta_{X,\mathcal V}-\lambda)u_f(\lambda)=0.
\)
The assumption
\(\lambda\in\rho(\Delta_{X,\mathcal V})\) gives
\(u_f(\lambda)=0\). Consequently,
\(f=\widetilde{\pi}_{\mathcal W^\#}(u_f(\lambda))=0\).
Thus \(T_{\mathcal V,\mathcal W}(\lambda)\) is injective and hence
an isomorphism.
\end{proof}

\begin{thm}[Kreĭn's resolvent formula]
\label{thm:Krein-resolvent-formula}
Let \((\mathcal V,\mathcal V^\#)\) and
\((\mathcal W,\mathcal W^\#)\) be Lagrangian pairs. If
\(\lambda\in
\rho(\Delta_{X,\mathcal V})
\cap\rho(\Delta_{X,\mathcal W})\), then
\[
R_{\mathcal V}(\lambda)-R_{\mathcal W}(\lambda)
=
-\Gamma_{\mathcal W,\mathcal W^\#}^{X}(\lambda)
T_{\mathcal V,\mathcal W}(\lambda)^{-1}
C_{\mathcal V}^{\dagger_{\omega_S}}
\widetilde{\pi}_{\mathcal W}
R_{\mathcal W}(\lambda).
\]
\end{thm}

\begin{proof}
Let \(F\in L^2(X')\), and set
\(v_\lambda:=R_{\mathcal V}(\lambda)F\in D_{\mathcal V}\) and
\(w_\lambda:=R_{\mathcal W}(\lambda)F\in D_{\mathcal W}\).
Set \(u_\lambda:=v_\lambda-w_\lambda\). Since
\((\lambda-\Delta_{X,\mathcal V})v_\lambda=F\) and
\((\lambda-\Delta_{X,\mathcal W})w_\lambda=F\), we have
\((\Delta_X-\lambda)u_\lambda=0\). Thus \(u_\lambda\) is a formal
solution at \(\lambda\).

Since \(\lambda\in\rho(\Delta_{X,\mathcal W})\), there exists a
unique \(f\in\mathcal W^\#\) such that
\(
u_\lambda
=
\Gamma_{\mathcal W,\mathcal W^\#}^{X}(\lambda)f.
\)
Consequently,
\(
\pi_S(u_\lambda)
=
M_{\mathcal W,\mathcal W^\#}^{X}(\lambda)f+f.
\)
Since \(v_\lambda=w_\lambda+u_\lambda\), it follows that
\[
\pi_S(v_\lambda)
=
\pi_S(w_\lambda)
+
M_{\mathcal W,\mathcal W^\#}^{X}(\lambda)f
+
f.
\]

Now \(\pi_S(w_\lambda)\in\mathcal W\), whereas
\(\pi_S(v_\lambda)\in\mathcal V\). Applying the boundary equation
for \(\mathcal V\) therefore gives
\(
T_{\mathcal V,\mathcal W}(\lambda)f
=
-C_{\mathcal V}^{\dagger_{\omega_S}}
\widetilde{\pi}_{\mathcal W}(w_\lambda).
\)
By Lemma~\ref{lem:boundary-operator-invertible},
\(T_{\mathcal V,\mathcal W}(\lambda)\) is invertible, and hence
\[
f
=
-T_{\mathcal V,\mathcal W}(\lambda)^{-1}
C_{\mathcal V}^{\dagger_{\omega_S}}
\widetilde{\pi}_{\mathcal W}(w_\lambda).
\]
Substituting this identity into
\(
u_\lambda
=
\Gamma_{\mathcal W,\mathcal W^\#}^{X}(\lambda)f
\)
and using \(w_\lambda=R_{\mathcal W}(\lambda)F\) gives the asserted
identity applied to \(F\). Since \(F\in L^2(X')\) is arbitrary, the
operator identity follows.
\end{proof}

Thus the difference between the resolvents of two self-adjoint
realizations is determined entirely by the finite-dimensional
critical asymptotic boundary data and the corresponding Weyl
function. In particular,
\(R_{\mathcal V}(\lambda)-R_{\mathcal W}(\lambda)\) is a
finite-rank operator of rank at most
\(\dim_{\mb C}\mathcal W^\#
=\frac12\dim_{\mb C}V_S\), and is therefore trace class.

\begin{lem}\label{lem:boundary-resolvent-M-derivative}
Let \(\lambda\in\rho(\Delta_{X,\mathcal W})\). Then
\[
\widetilde{\pi}_{\mathcal W}
R_{\mathcal W}(\lambda)
\Gamma_{\mathcal W,\mathcal W^\#}^{X}(\lambda)
=
-\frac{d}{d\lambda}
M_{\mathcal W,\mathcal W^\#}^{X}(\lambda)
\]
as linear maps from \(\mathcal W^\#\) to \(\mathcal W\).
\end{lem}

\begin{proof}
Let \(f\in\mathcal W^\#\), and set
\(u_f(\lambda):=
\Gamma_{\mathcal W,\mathcal W^\#}^{X}(\lambda)f\).
Then
\((\Delta_X-\lambda)u_f(\lambda)=0\) and
\(
\pi_S(u_f(\lambda))
=
M_{\mathcal W,\mathcal W^\#}^{X}(\lambda)f+f.
\)

Differentiating with respect to \(\lambda\), we obtain
\[
    (\Delta_X-\lambda)u_f'(\lambda)=u_f(\lambda)
\]
and
\[
\pi_S(u_f'(\lambda))
=
\frac{d}{d\lambda}
M_{\mathcal W,\mathcal W^\#}^{X}(\lambda)f.
\]
In particular, the boundary data of \(u_f'(\lambda)\) belong to
\(\mathcal W\), so \(u_f'(\lambda)\in D_{\mathcal W}\).

Since
\(
(\lambda-\Delta_{X,\mathcal W})u_f'(\lambda)
=
-u_f(\lambda),
\)
it follows that
\(
u_f'(\lambda)
=
-R_{\mathcal W}(\lambda)u_f(\lambda).
\)
Applying \(\widetilde{\pi}_{\mathcal W}\) gives
\[
\frac{d}{d\lambda}
M_{\mathcal W,\mathcal W^\#}^{X}(\lambda)f
=
-\widetilde{\pi}_{\mathcal W}
R_{\mathcal W}(\lambda)
\Gamma_{\mathcal W,\mathcal W^\#}^{X}(\lambda)f.
\]
Since this holds for every \(f\in\mathcal W^\#\), the result follows.
\end{proof}

Choose fixed bases of \(\mathcal W^\#\) and \(\mathcal V^\#\), and
write \(\det T_{\mathcal V,\mathcal W}(\lambda)\) for the
determinant of the corresponding matrix. Although this determinant
depends on the chosen bases, its logarithmic derivative does not.

\begin{cor}\label{cor:resolvent-trace-identity}
If
\(\lambda\in
\rho(\Delta_{X,\mathcal V})
\cap\rho(\Delta_{X,\mathcal W})\), then
\[
\begin{aligned}
\operatorname{Tr}_{L^2(X')}
\bigl(R_{\mathcal V}(\lambda)-R_{\mathcal W}(\lambda)\bigr)
&=
\operatorname{Tr}_{\mathcal W^\#}
\bigl(
T_{\mathcal V,\mathcal W}(\lambda)^{-1}
T_{\mathcal V,\mathcal W}'(\lambda)
\bigr) \\
&=
\frac{d}{d\lambda}
\log\det T_{\mathcal V,\mathcal W}(\lambda).
\end{aligned}
\]
\end{cor}

\begin{proof}
The resolvent difference is finite rank, so its trace is
well-defined. By Kreĭn's resolvent formula and the cyclicity of the
trace for finite-rank factorizations,
\[
\begin{aligned}
\operatorname{Tr}_{L^2(X')}
\bigl(R_{\mathcal V}(\lambda)-R_{\mathcal W}(\lambda)\bigr)
&=
-\operatorname{Tr}_{\mathcal W^\#}\bigl(
T_{\mathcal V,\mathcal W}(\lambda)^{-1}
C_{\mathcal V}^{\dagger_{\omega_S}}
\widetilde{\pi}_{\mathcal W}
R_{\mathcal W}(\lambda)
\Gamma_{\mathcal W,\mathcal W^\#}^{X}(\lambda)
\bigr) \\
&=
\operatorname{Tr}_{\mathcal W^\#}\bigl(
T_{\mathcal V,\mathcal W}(\lambda)^{-1}
C_{\mathcal V}^{\dagger_{\omega_S}}
\frac{d}{d\lambda}
M_{\mathcal W,\mathcal W^\#}^{X}(\lambda)
\bigr),
\end{aligned}
\]
where Lemma~\ref{lem:boundary-resolvent-M-derivative} was used in
the second equality. Since
\(
T_{\mathcal V,\mathcal W}'(\lambda)
=
C_{\mathcal V}^{\dagger_{\omega_S}}
\frac{d}{d\lambda}
M_{\mathcal W,\mathcal W^\#}^{X}(\lambda),
\)
this proves the first identity.

The map \(T_{\mathcal V,\mathcal W}(\lambda)\) is holomorphic on
\(\rho(\Delta_{X,\mathcal W})\) and invertible on
\(\rho(\Delta_{X,\mathcal V})
\cap\rho(\Delta_{X,\mathcal W})\).
Jacobi's formula therefore gives the second identity.
\end{proof}
Fix bases of \(\mathcal W^\#\) and \(\mathcal V^\#\). Define
\(d_{(\mathcal V,\mathcal V^\#),(\mathcal W,\mathcal W^\#)}^{X}
(\lambda)\) to be the determinant of the matrix representing
\(T_{(\mathcal V,\mathcal V^\#),(\mathcal W,\mathcal W^\#)}
(\lambda)\) with respect to these bases. Once the two Lagrangian
pairs and the bases have been fixed, we abbreviate this function as
\(
d_{\mathcal V,\mathcal W}^{X}(\lambda)
:=
\det T_{\mathcal V,\mathcal W}(\lambda).
\)

Changing either basis multiplies
\(d_{\mathcal V,\mathcal W}^{X}\) by a nonzero constant independent
of \(\lambda\). Hence its logarithmic derivative and divisor are
independent of the chosen bases.

\begin{prop}\label{prop:order-of-determinant-function}
The function \(d_{\mathcal V,\mathcal W}^{X}\) extends
meromorphically to \(\mb C\). If \(\lambda_0\) is an eigenvalue of
either \(\Delta_{X,\mathcal V}\) or
\(\Delta_{X,\mathcal W}\), then
\[
\operatorname{ord}_{\lambda_0}
d_{\mathcal V,\mathcal W}^{X}
=
\dim_{\mb C}\ker(\Delta_{X,\mathcal V}-\lambda_0)
-
\dim_{\mb C}\ker(\Delta_{X,\mathcal W}-\lambda_0).
\]
\end{prop}

\begin{proof}
Since \(T_{\mathcal V,\mathcal W}\) is meromorphic,
\(
d_{\mathcal V,\mathcal W}^{X}
=
\det T_{\mathcal V,\mathcal W}
\)
extends meromorphically to \(\mb C\). Moreover, for
\(\lambda\in
\rho(\Delta_{X,\mathcal V})
\cap\rho(\Delta_{X,\mathcal W})\),
Corollary~\ref{cor:resolvent-trace-identity} gives
\[
\operatorname{Tr}_{L^2(X')}
\bigl(
R_{\mathcal V}(\lambda)-R_{\mathcal W}(\lambda)
\bigr)
=
\frac{d}{d\lambda}
\log d_{\mathcal V,\mathcal W}^{X}(\lambda).
\]

Fix \(\lambda_0\), and denote by
\(\Pi_{\mathcal V,\lambda_0}\) and
\(\Pi_{\mathcal W,\lambda_0}\) the orthogonal projections onto
\(\ker(\Delta_{X,\mathcal V}-\lambda_0)\) and
\(\ker(\Delta_{X,\mathcal W}-\lambda_0)\), respectively. If
\(\lambda_0\) is not an eigenvalue of one of the operators, take the
corresponding projection to be zero. Since
\(R_{\mathcal V}(\lambda)
=(\lambda-\Delta_{X,\mathcal V})^{-1}\), subtracting the two
Laurent expansions gives
\[
R_{\mathcal V}(\lambda)-R_{\mathcal W}(\lambda)
=
\frac{
\Pi_{\mathcal V,\lambda_0}
-
\Pi_{\mathcal W,\lambda_0}
}{
\lambda-\lambda_0
}
+
H(\lambda),
\]
where \(H\) is a holomorphic finite-rank operator-valued function
near \(\lambda_0\).

Taking the trace and comparing the residue at \(\lambda_0\) with
that of
\(
(d_{\mathcal V,\mathcal W}^{X})'/
d_{\mathcal V,\mathcal W}^{X}
\)
gives
\[
\operatorname{ord}_{\lambda_0}
d_{\mathcal V,\mathcal W}^{X}
=
\dim_{\mb C}\ker(\Delta_{X,\mathcal V}-\lambda_0)
-
\dim_{\mb C}\ker(\Delta_{X,\mathcal W}-\lambda_0).
\]
\end{proof}

\section{A Comparison Formula for Zeta Determinants}

We compare the positive-spectrum zeta determinants of self-adjoint
realizations. Since these realizations need not be nonnegative, all
nonpositive eigenvalues are omitted, producing a finite-dimensional
correction factor. A relative logarithmic derivative formula,
combined with Kreĭn's trace identity, gives the comparison in terms of
\(d_{\mathcal V,\mathcal W}^{X}(\mu)\) and
\(E_{\mathcal V,\mathcal W}^{X}(\mu)\), subject to the following
standing analytic hypothesis.

\begin{ass}[Zeta-regularity hypothesis]\label{ass:zeta-regularity}
For every pair of self-adjoint realizations to which the comparison
formula is applied, the corresponding positive-spectrum zeta functions admit
meromorphic continuations to a neighborhood of \(s=0\), are regular
there, and depend differentiably on the spectral shift. We also
assume that the relative contour representation has sufficient
locally uniform trace-norm decay to justify differentiation with
respect to the shift, integration by parts on truncated contours, and
passage to the limit.
\end{ass}

\begin{rem}
The individual zeta-regularity requirement in
Assumption~\ref{ass:zeta-regularity} holds for the Friedrichs
realization in the standard conic heat-kernel framework; see, for
example, \cite{LP}. It is not automatic for a general self-adjoint
realization. General
Lagrangian boundary conditions may produce nonstandard poles or
logarithmic singularities of the spectral zeta function, possibly at
\(s=0\); see \cite{KirstenLoyaParkExotic}. Accordingly, the ordinary
zeta-determinant comparison below is stated only for pairs of
realizations satisfying Assumption~\ref{ass:zeta-regularity}. A
treatment of arbitrary self-adjoint realizations would in general
require a suitably modified determinant. Thus
Assumption~\ref{ass:zeta-regularity} is an independent analytic input,
not a consequence of the finite-dimensional extension theory.
\end{rem}

Let \(A:D(A)\subset\mathcal H\to\mathcal H\) be a self-adjoint
operator on a separable infinite-dimensional Hilbert space
\(\mathcal H\). Assume that \(A\) is bounded from below and has
compact resolvent. Its spectrum is then discrete and consists of
real eigenvalues of finite multiplicity. We enumerate them, with
multiplicity, as
\(\lambda_1\leq\lambda_2\leq\cdots\), where
\(\lambda_k\to+\infty\).
Assume in addition that
\(\sum_{\lambda_k>0}\lambda_k^{-s}\) converges for
\(\operatorname{Re}s\) sufficiently large. The positive-spectrum
zeta function of \(A\) is then defined in this right half-plane by
\[
\zeta_A(s):=
\sum_{\lambda_k>0}\lambda_k^{-s}.
\]
Thus all nonpositive eigenvalues are omitted.
Let \(P_A^+:\mathcal H\to\mathcal H\) be the spectral projection
onto the positive spectral subspace of \(A\). For
\(\operatorname{Re}s\) sufficiently large, the Mellin transform
gives
\[
\zeta_A(s)
=
\frac{1}{\Gamma(s)}
\int_0^\infty
t^{s-1}
\operatorname{Tr}_{\mathcal H}
\bigl(e^{-tA}P_A^+\bigr)\,dt.
\]
If \(A\) is nonnegative, then
\(P_A^+=I-P_{\ker A}\), where
\(P_{\ker A}:\mathcal H\to\mathcal H\) is the orthogonal projection
onto \(\ker A\). In this case,
\[
\zeta_A(s)
=
\frac{1}{\Gamma(s)}
\int_0^\infty
t^{s-1}
\operatorname{Tr}_{\mathcal H}
\bigl(e^{-tA}-P_{\ker A}\bigr)\,dt.
\]
Suppose that \(\zeta_A\) admits a meromorphic continuation to a
neighborhood of \(s=0\) and is regular there. The positive-spectrum
zeta determinant of \(A\) is defined by
\[
\operatorname{Det}_{\zeta}(A)
:=
\exp\bigl(-\zeta_A'(0)\bigr).
\]
Throughout this section,
\(\operatorname{Det}_{\zeta}(A)\) always refers to this
positive-spectrum convention.

Now let \(A:D(A)\subset\mathcal H\to\mathcal H\) and
\(B:D(B)\subset\mathcal H\to\mathcal H\) be self-adjoint operators
satisfying the preceding assumptions. Define their relative zeta
function by
\(\zeta_{(A,B)}(s):=\zeta_A(s)-\zeta_B(s)\), and define their
relative zeta determinant by
\(\operatorname{Det}_{\zeta}(A,B)
:=\exp(-\zeta_{(A,B)}'(0))\).
It follows immediately that
\(\operatorname{Det}_{\zeta}(A,B)
=\operatorname{Det}_{\zeta}(A)/
\operatorname{Det}_{\zeta}(B)\).

Write \(R_A(z):=(z-A)^{-1}\) and \(R_B(z):=(z-B)^{-1}\).

Let \(\gamma_{A,B}^+\) be an admissible positively oriented spectral
contour surrounding the positive spectra of \(A\) and \(B\), while
excluding their nonpositive spectra. Since the positive spectra are
unbounded, this contour is understood as an unbounded spectral
contour, or equivalently as the limit of suitable truncated
contours. We choose the branch of \(z^{-s}\) on
\(\mathbb C\setminus(-\infty,0]\).

Assume that \(R_A(z)-R_B(z)\) is trace class along the contour and
has sufficient decay at infinity for the following integral to
converge when \(\operatorname{Re}s\) is sufficiently large. Then
\[
\zeta_{(A,B)}(s)
=
\frac{1}{2\pi i}
\int_{\gamma_{A,B}^+}
z^{-s}
\operatorname{Tr}_{\mathcal H}
\bigl(R_A(z)-R_B(z)\bigr)\,dz.
\]

Suppose further that the meromorphic continuation of this contour
integral is regular at \(s=0\). By definition,
\[
\log\operatorname{Det}_{\zeta}(A,B)
=
-\left.
\frac{d}{ds}
\right|_{s=0}
\left[
\frac{1}{2\pi i}
\int_{\gamma_{A,B}^+}
z^{-s}
\operatorname{Tr}_{\mathcal H}
\bigl(R_A(z)-R_B(z)\bigr)\,dz
\right]^{\mathrm{mer}}.
\]
If differentiation under the continued integral is justified at
\(s=0\), this becomes
\[
\log\operatorname{Det}_{\zeta}(A,B)
=
\frac{1}{2\pi i}
\int_{\gamma_{A,B}^+}
\log z\,
\operatorname{Tr}_{\mathcal H}
\bigl(R_A(z)-R_B(z)\bigr)\,dz.
\]
For a real parameter
\(\mu\in\rho(A)\cap\rho(B)\), consider the shifted pair
\((A-\mu,B-\mu)\). Under the same assumptions,
\[
\begin{aligned}
\log\operatorname{Det}_{\zeta}(A-\mu,B-\mu)
&=
\frac{1}{2\pi i}
\int_{\gamma_\mu^+}
\log z\,
\operatorname{Tr}_{\mathcal H}\bigl(
(z-A+\mu)^{-1}-(z-B+\mu)^{-1}
\bigr)\,dz,
\end{aligned}
\]
where \(\gamma_\mu^+\) surrounds the positive spectra of
\(A-\mu\) and \(B-\mu\) and excludes their nonpositive spectra.

After shrinking the interval of \(\mu\), if necessary, we may
assume that no eigenvalue of \(A-\mu\) or \(B-\mu\) crosses zero
and choose the admissible contour \(\gamma^+\) independently of
\(\mu\). Set
\(Q_A(z,\mu):=(z-A+\mu)^{-1}\) and
\(Q_B(z,\mu):=(z-B+\mu)^{-1}\).
Then
\(\partial_\mu Q_A(z,\mu)
=\partial_zQ_A(z,\mu)=-Q_A(z,\mu)^2\), and similarly for \(B\).

Differentiate first on the truncated contours and integrate by parts
there. By the assumed trace-norm decay of
\(Q_A(z,\mu)-Q_B(z,\mu)\), the boundary terms at the truncation
points tend to zero, locally uniformly for \(\mu\) in the interval
under consideration. Passing to the limit gives
\[
\begin{aligned}
\frac{d}{d\mu}
\log\operatorname{Det}_{\zeta}(A-\mu,B-\mu)
&=
-\frac{1}{2\pi i}
\int_{\gamma^+}
\frac{1}{z}
\operatorname{Tr}_{\mathcal H}\bigl(
Q_A(z,\mu)-Q_B(z,\mu)
\bigr)\,dz.
\end{aligned}
\]

Let \(P_{A-\mu}^+\) and \(P_{B-\mu}^+\) be the spectral projections
onto the positive spectral subspaces of \(A-\mu\) and \(B-\mu\),
respectively. Applying the holomorphic functional calculus to the
relative resolvent on truncated contours and then passing to the limit
under the standing trace-class and decay assumptions gives
\[
\begin{aligned}
&\frac{1}{2\pi i}
\int_{\gamma^+}
\frac{1}{z}
\operatorname{Tr}_{\mathcal H}\bigl(
Q_A(z,\mu)-Q_B(z,\mu)
\bigr)\,dz \\
&\qquad=
\operatorname{Tr}_{\mathcal H}\left(
(A-\mu)^{-1}P_{A-\mu}^+
-
(B-\mu)^{-1}P_{B-\mu}^+
\right).
\end{aligned}
\]
Here only the relative difference on the right is asserted to be trace
class; its two terms need not be trace class separately. Consequently,
\[
\begin{aligned}
\frac{d}{d\mu}
\log\operatorname{Det}_{\zeta}(A-\mu,B-\mu)
&=
-\operatorname{Tr}_{\mathcal H}\bigl(
(A-\mu)^{-1}P_{A-\mu}^+
-
(B-\mu)^{-1}P_{B-\mu}^+
\bigr).
\end{aligned}
\]
We denote the trace on the right-hand side by
\[
\operatorname{Tr}^+\bigl(
(A-\mu)^{-1}-(B-\mu)^{-1}
\bigr)
:=
\operatorname{Tr}_{\mathcal H}\bigl(
(A-\mu)^{-1}P_{A-\mu}^+
-
(B-\mu)^{-1}P_{B-\mu}^+
\bigr).
\]
Thus
\[
\frac{d}{d\mu}
\log\operatorname{Det}_{\zeta}(A-\mu,B-\mu)
=
-\operatorname{Tr}^+\bigl(
(A-\mu)^{-1}-(B-\mu)^{-1}
\bigr).
\]
Let \(P_{A-\mu}^-\) and \(P_{B-\mu}^-\) denote the spectral
projections onto the negative spectral subspaces of \(A-\mu\) and
\(B-\mu\), respectively. Define
\(\operatorname{Tr}^-\bigl((A-\mu)^{-1}-(B-\mu)^{-1}\bigr)\) by
\[
\operatorname{Tr}_{\mathcal H}\bigl(
(A-\mu)^{-1}P_{A-\mu}^-
-
(B-\mu)^{-1}P_{B-\mu}^-
\bigr).
\]
These operators have finite rank because \(A\) and \(B\) are
bounded from below. Since \(\mu\in\rho(A)\cap\rho(B)\), zero is not
an eigenvalue of either shifted operator. Hence
\[
\begin{aligned}
\operatorname{Tr}^+\bigl(
(A-\mu)^{-1}-(B-\mu)^{-1}
\bigr)
&=
\operatorname{Tr}_{\mathcal H}\bigl(
(A-\mu)^{-1}-(B-\mu)^{-1}
\bigr) \\
&\quad-
\operatorname{Tr}^-\bigl(
(A-\mu)^{-1}-(B-\mu)^{-1}
\bigr).
\end{aligned}
\]

Let \(I\subset\mathbb R\cap\rho(A)\cap\rho(B)\) be an open interval.
For \(\mu\in I\), define
\[
E_{A,B}(\mu)
:=
\frac{
\displaystyle
\prod_{\mu-\lambda_\ell^A>0}
(\mu-\lambda_\ell^A)^{m(\lambda_\ell^A)}
}{
\displaystyle
\prod_{\mu-\lambda_\ell^B>0}
(\mu-\lambda_\ell^B)^{m(\lambda_\ell^B)}
}.
\]
The products are taken over the distinct eigenvalues satisfying the
indicated inequalities. They are finite, and an empty product is
understood to be \(1\). Since every factor is positive,
\(E_{A,B}(\mu)>0\).

Since no eigenvalue crosses \(\mu\) while \(\mu\) remains in \(I\),
the sets of eigenvalues appearing in the products are constant on
\(I\). Therefore,
\[
\frac{d}{d\mu}\log E_{A,B}(\mu)
=
\sum_{\mu-\lambda_\ell^A>0}
\frac{m(\lambda_\ell^A)}{\mu-\lambda_\ell^A}
-
\sum_{\mu-\lambda_\ell^B>0}
\frac{m(\lambda_\ell^B)}{\mu-\lambda_\ell^B}.
\]
\begin{thm}\label{thm:log-derivative-relative-zeta-determinant}
Let \(A\) and \(B\) be self-adjoint operators on \(\mathcal H\),
bounded from below and having compact resolvent. Let
\(I\subset\mathbb R\cap\rho(A)\cap\rho(B)\) be an open interval.
Assume that, for every \(\mu\in I\), the positive-spectrum zeta
functions of \(A-\mu\) and \(B-\mu\) admit meromorphic
continuations that are regular at \(s=0\) and depend
differentiably on \(\mu\) near \(s=0\). Assume also that
\((A-\mu)^{-1}-(B-\mu)^{-1}\) is trace class for every
\(\mu\in I\). Finally, assume that the relative contour
representation above holds with sufficient trace-norm decay,
locally uniformly in \(\mu\), to justify differentiation with
respect to \(\mu\), integration by parts on truncated contours, and
passage to the limit.

Then
\[
\begin{aligned}
\frac{d}{d\mu}
\log\operatorname{Det}_{\zeta}(A-\mu,B-\mu)
&=
-\operatorname{Tr}_{\mathcal H}\bigl(
(A-\mu)^{-1}-(B-\mu)^{-1}
\bigr)
-
\frac{d}{d\mu}\log E_{A,B}(\mu) \\
&=
\operatorname{Tr}_{\mathcal H}\bigl(
R_A(\mu)-R_B(\mu)
\bigr)
-
\frac{d}{d\mu}\log E_{A,B}(\mu).
\end{aligned}
\]
\end{thm}

We now apply the relative formula to self-adjoint realizations of the
conic Laplacian.
Let \((\mathcal V,\mathcal V^\#)\) and
\((\mathcal W,\mathcal W^\#)\) be Lagrangian pairs in
\((V_S,\omega_S)\). Let
\(I\subset\mb R\cap
\rho(\Delta_{X,\mathcal V})
\cap\rho(\Delta_{X,\mathcal W})\)
be an open interval, and set
\(
E_{\mathcal V,\mathcal W}^{X}(\mu)
:=
E_{\Delta_{X,\mathcal V},\Delta_{X,\mathcal W}}(\mu)
\)
for \(\mu\in I\).

We verify the operator-theoretic hypotheses of
Theorem~\ref{thm:log-derivative-relative-zeta-determinant} in the
present setting. The realizations \(\Delta_{X,\mathcal V}\) and
\(\Delta_{X,\mathcal W}\) are bounded from below and have compact
resolvent, while Kreĭn's resolvent formula shows that
\(R_{\mathcal V}(\mu)-R_{\mathcal W}(\mu)\) is finite rank. Since
\[
(\Delta_{X,\mathcal V}-\mu)^{-1}
-
(\Delta_{X,\mathcal W}-\mu)^{-1}
=
-\bigl(
R_{\mathcal V}(\mu)-R_{\mathcal W}(\mu)
\bigr),
\]
the difference of the shifted inverses is also finite rank and
hence trace class. The remaining analytic hypotheses are precisely
Assumption~\ref{ass:zeta-regularity}.
Theorem~\ref{thm:log-derivative-relative-zeta-determinant} expresses
the logarithmic derivative of the relative zeta determinant as the
trace of the resolvent difference minus the logarithmic derivative
of \(E_{\mathcal V,\mathcal W}^{X}\). By
Corollary~\ref{cor:resolvent-trace-identity}, this resolvent trace is
\(
\frac{d}{d\mu}
\log d_{\mathcal V,\mathcal W}^{X}(\mu).
\)
Therefore,
\[
\frac{d}{d\mu}
\log\operatorname{Det}_{\zeta}
(\Delta_{X,\mathcal V}-\mu,\Delta_{X,\mathcal W}-\mu)
=
\frac{d}{d\mu}
\log d_{\mathcal V,\mathcal W}^{X}(\mu)
-
\frac{d}{d\mu}
\log E_{\mathcal V,\mathcal W}^{X}(\mu).
\]
Equivalently,
\[
\frac{d}{d\mu}
\log\left(
\frac{
\operatorname{Det}_{\zeta}
(\Delta_{X,\mathcal V}-\mu,\Delta_{X,\mathcal W}-\mu)
E_{\mathcal V,\mathcal W}^{X}(\mu)
}{
d_{\mathcal V,\mathcal W}^{X}(\mu)
}
\right)
=
0.
\]
Hence the quantity inside the logarithm is constant on \(I\), and
the desired comparison formula follows.

\begin{thm}\label{thm:shifted-comparison-formula}
Assume Assumption~\ref{ass:zeta-regularity}.
Let \(I\) be an open
interval contained in
\(
\mb R\setminus
(\sigma(\Delta_{X,\mathcal V})
\cup\sigma(\Delta_{X,\mathcal W})).
\)
Then there exists a nonzero constant
\(C_{\mathcal V,\mathcal W}\), independent of \(\mu\in I\), such
that
\[
\frac{
\operatorname{Det}_{\zeta}(\Delta_{X,\mathcal V}-\mu)
}{
\operatorname{Det}_{\zeta}(\Delta_{X,\mathcal W}-\mu)
}
=
C_{\mathcal V,\mathcal W}
\frac{
d_{\mathcal V,\mathcal W}^{X}(\mu)
}{
E_{\mathcal V,\mathcal W}^{X}(\mu)
}.
\]
The constant may depend on \(I\), the Lagrangian pairs
\((\mathcal V,\mathcal V^\#)\) and
\((\mathcal W,\mathcal W^\#)\), and the bases used to define
\(d_{\mathcal V,\mathcal W}^{X}\).
\end{thm}

\begin{rem}
The extension-theoretic, compactness, and trace-class ingredients of
the theorem have been established in the preceding sections. The
conditional input is confined to the zeta continuation and relative
contour estimates stated in Assumption~\ref{ass:zeta-regularity}.
\end{rem}

\begin{cor}\label{cor:zeta-determinant-comparison-zero}
Assume Assumption~\ref{ass:zeta-regularity}. Set
\(
r:=
\dim_{\mb C}\ker\Delta_{X,\mathcal V}
-
\dim_{\mb C}\ker\Delta_{X,\mathcal W},
\)
and define
\(
(d_{\mathcal V,\mathcal W}^{X})^{\mathrm{red}}(\lambda)
:=
\lambda^{-r}d_{\mathcal V,\mathcal W}^{X}(\lambda).
\)
Then
\((d_{\mathcal V,\mathcal W}^{X})^{\mathrm{red}}\) is holomorphic
and nonzero at \(\lambda=0\), and
\[
\frac{
\operatorname{Det}_{\zeta}(\Delta_{X,\mathcal V})
}{
\operatorname{Det}_{\zeta}(\Delta_{X,\mathcal W})
}
=
C_{\mathcal V,\mathcal W}
(d_{\mathcal V,\mathcal W}^{X})^{\mathrm{red}}(0)
\frac{
\displaystyle
\prod_{\lambda_\ell^{\mathcal W}<0}
(-\lambda_\ell^{\mathcal W})
^{m(\lambda_\ell^{\mathcal W})}
}{
\displaystyle
\prod_{\lambda_\ell^{\mathcal V}<0}
(-\lambda_\ell^{\mathcal V})
^{m(\lambda_\ell^{\mathcal V})}
}.
\]
The products are taken over the distinct negative eigenvalues, with
their multiplicities appearing as exponents, and an empty product
is understood to be \(1\).
\end{cor}

\begin{proof}
Choose \(\varepsilon>0\) such that
\(
(-\varepsilon,\varepsilon)\cap
\bigl(
\sigma(\Delta_{X,\mathcal V})
\cup\sigma(\Delta_{X,\mathcal W})
\bigr)
\subset\{0\},
\)
and apply Theorem~\ref{thm:shifted-comparison-formula} on
\(I=(0,\varepsilon)\).

For \(\mu\in I\), every zero eigenvalue becomes the negative
eigenvalue \(-\mu\) of the corresponding shifted operator. By the
definition of \(E_{\mathcal V,\mathcal W}^{X}\),
\(
E_{\mathcal V,\mathcal W}^{X}(\mu)
=
\mu^r\widetilde E_{\mathcal V,\mathcal W}^{X}(\mu),
\)
where
\[
\widetilde E_{\mathcal V,\mathcal W}^{X}(\mu)
:=
\frac{
\displaystyle
\prod_{\lambda_\ell^{\mathcal V}<0}
(\mu-\lambda_\ell^{\mathcal V})
^{m(\lambda_\ell^{\mathcal V})}
}{
\displaystyle
\prod_{\lambda_\ell^{\mathcal W}<0}
(\mu-\lambda_\ell^{\mathcal W})
^{m(\lambda_\ell^{\mathcal W})}
}.
\]

By Proposition~\ref{prop:order-of-determinant-function},
\(
\operatorname{ord}_0
d_{\mathcal V,\mathcal W}^{X}
=
r.
\)
Hence
\(
d_{\mathcal V,\mathcal W}^{X}(\mu)
=
\mu^r
(d_{\mathcal V,\mathcal W}^{X})^{\mathrm{red}}(\mu).
\)
The factors \(\mu^r\) cancel in the shifted comparison formula,
giving
\[
\frac{
\operatorname{Det}_{\zeta}(\Delta_{X,\mathcal V}-\mu)
}{
\operatorname{Det}_{\zeta}(\Delta_{X,\mathcal W}-\mu)
}
=
C_{\mathcal V,\mathcal W}
\frac{
(d_{\mathcal V,\mathcal W}^{X})^{\mathrm{red}}(\mu)
}{
\widetilde E_{\mathcal V,\mathcal W}^{X}(\mu)
}.
\]

Letting \(\mu\to0^+\), the positive-spectrum convention shows that
the zeta determinants of the shifted operators converge to the
corresponding unshifted positive-spectrum zeta determinants. Moreover,
\[
\widetilde E_{\mathcal V,\mathcal W}^{X}(0)
=
\frac{
\displaystyle
\prod_{\lambda_\ell^{\mathcal V}<0}
(-\lambda_\ell^{\mathcal V})
^{m(\lambda_\ell^{\mathcal V})}
}{
\displaystyle
\prod_{\lambda_\ell^{\mathcal W}<0}
(-\lambda_\ell^{\mathcal W})
^{m(\lambda_\ell^{\mathcal W})}
}.
\]
The result follows.
\end{proof}

\section{The Weyl Curve and the Determinant Line}
We now organize the preceding analytic results geometrically. Related
maps into Grassmannians, formed from null spaces of elliptic cone
operators, appear in \cite{GilKrainerMendoza2007}, while Wang developed an abstract
theory of Weyl curves for symmetric operators
\cite{WangWeylCurve}. The treatment below is self-contained and is
formulated directly in terms of the critical asymptotic boundary data
constructed above. Its purpose is to interpret the analytic identities
from the preceding sections: local graph coordinates recover the Weyl
functions, the real tangent form recovers their derivative identity,
and the pullback determinant line gives a geometric interpretation of
the boundary and zeta determinants.

For \(\lambda\in\mb C\), define the intrinsic boundary-data
subspace
\(\mc K_X(\lambda):=
\pi_S\bigl(\ker(\Delta_X-\lambda)\bigr)\subset V_S\).

\begin{lem}\label{lem:cauchy-data-orthogonality}
For every \(\lambda\in\mb C\), one has
\(\mc K_X(\overline{\lambda})
\subset\mc K_X(\lambda)^{\perp_{\omega_S}}\).
In particular, if \(\lambda\in\mb R\), then
\(\mc K_X(\lambda)\) is isotropic.
\end{lem}

\begin{proof}
Let \(g\in\mc K_X(\overline{\lambda})\) and
\(h\in\mc K_X(\lambda)\). Choose
\(v\in\ker(\Delta_X-\overline{\lambda})\) and
\(u\in\ker(\Delta_X-\lambda)\) such that
\(\pi_Sv=g\) and \(\pi_Su=h\). By the Green identity,
\[
\begin{aligned}
\omega_S(g,h)
&=
q_X(v,u) \\
&=
\langle\Delta_Xv,u\rangle
-
\langle v,\Delta_Xu\rangle \\
&=
\overline{\lambda}\langle v,u\rangle
-
\overline{\lambda}\langle v,u\rangle
=
0.
\end{aligned}
\]
Here the second equality in the last line uses the convention that
the inner product is linear in the first variable and conjugate
linear in the second. Hence
\(g\in\mc K_X(\lambda)^{\perp_{\omega_S}}\), proving the asserted
inclusion.

If \(\lambda\in\mb R\), then
\(\overline{\lambda}=\lambda\), and hence
\(\mc K_X(\lambda)
\subset\mc K_X(\lambda)^{\perp_{\omega_S}}\).
Thus \(\mc K_X(\lambda)\) is isotropic.
\end{proof}

\begin{lem}\label{lem:cauchy-data-orthogonal-complement}
Let \(\mc V\in\operatorname{LGr}(V_S,\omega_S)\). If
\(\lambda\in\rho(\Delta_{X,\mc V})\), then
\(\mc K_X(\overline{\lambda})
=\mc K_X(\lambda)^{\perp_{\omega_S}}\).
In particular, if
\(\lambda\in\mb R\cap\rho(\Delta_{X,\mc V})\), then
\(\mc K_X(\lambda)\) is Lagrangian.
\end{lem}
\begin{proof}
By Lemma~\ref{lem:cauchy-data-orthogonality},
\(\mc K_X(\overline{\lambda})
\subset\mc K_X(\lambda)^{\perp_{\omega_S}}\).
Choose a Lagrangian complement \(\mc V^\#\) of \(\mc V\). Since
\(\Delta_{X,\mc V}\) is self-adjoint,
\(\lambda\in\rho(\Delta_{X,\mc V})\) implies
\(\overline{\lambda}\in\rho(\Delta_{X,\mc V})\). For
\(\mu\in\{\lambda,\overline{\lambda}\}\), the Poisson operator
\(\Gamma_{\mc V,\mc V^\#}^X(\mu):
\mc V^\#\to\ker(\Delta_X-\mu)\)
is an isomorphism.
Moreover, the restriction of \(\pi_S\) to
\(\ker(\Delta_X-\mu)\) is injective. Indeed, if
\(u\in\ker(\Delta_X-\mu)\) and \(\pi_Su=0\), then
\(u\in D_{\min}(\Delta_X)\subset D_{\mc V}\), and hence
\((\Delta_{X,\mc V}-\mu)u=0\). Since
\(\mu\in\rho(\Delta_{X,\mc V})\), it follows that \(u=0\).
By the definition of \(\mc K_X(\mu)\), this restriction is also
surjective onto \(\mc K_X(\mu)\), and is therefore an isomorphism.
Consequently,
\[
\dim_{\mb C}\mc K_X(\lambda)
=
\dim_{\mb C}\mc K_X(\overline{\lambda})
=
\dim_{\mb C}\mc V^\#
=
\frac12\dim_{\mb C}V_S.
\]
Since \(\omega_S\) is nondegenerate,
\(\dim_{\mb C}\mc K_X(\lambda)^{\perp_{\omega_S}}
=\dim_{\mb C}V_S-\dim_{\mb C}\mc K_X(\lambda)
=\frac12\dim_{\mb C}V_S\).
The preceding inclusion is therefore an equality.

If \(\lambda\in\mb R\), then
\(\overline{\lambda}=\lambda\), and hence
\(\mc K_X(\lambda)
=\mc K_X(\lambda)^{\perp_{\omega_S}}\).
Thus \(\mc K_X(\lambda)\) is Lagrangian.
\end{proof}

Set
\(
    \Sigma_X
    :=
    \bigcup_{\mc V\in\operatorname{LGr}(V_S,\omega_S)}
    \rho(\Delta_{X,\mc V}).
\)

\begin{prop}\label{prop:domain-of-cauchy-data-map}
We have
\[
    \Sigma_X
    =
    \mb C\setminus\sigma_p(\Delta_{X,\min}),
\]
where \(\sigma_p(\Delta_{X,\min})\) denotes the point spectrum of the
minimal Laplacian. Consequently, \(\Sigma_X\) is open and
path-connected.
\end{prop}

\begin{proof}
Since
\(
    \Sigma_X
    =
    \bigcup_{\mc V\in\operatorname{LGr}(V_S,\omega_S)}
    \rho(\Delta_{X,\mc V}),
\)
and each resolvent set is open, the set \(\Sigma_X\) is open.
Suppose first that
\(\lambda_0\in\sigma_p(\Delta_{X,\min})\). Then there exists
\(0\neq u_0\in D_{\min}(\Delta_X)\) such that
\(
    (\Delta_{X,\min}-\lambda_0)u_0=0.
\)
For every
\(\mc V\in\operatorname{LGr}(V_S,\omega_S)\), we have
\(D_{\min}(\Delta_X)\subset D_{\mc V}\), and hence
\(
    (\Delta_{X,\mc V}-\lambda_0)u_0=0.
\)
Thus
\(\lambda_0\in\sigma_p(\Delta_{X,\mc V})\) for every \(\mc V\), so
\(\lambda_0\notin\Sigma_X\). Therefore
\[
    \sigma_p(\Delta_{X,\min})
    \subset
    \mb C\setminus\Sigma_X.
\]
Conversely, suppose that
\(\lambda_0\notin\sigma_p(\Delta_{X,\min})\). We show that
\(\lambda_0\in\Sigma_X\).
If \(\lambda_0\notin\mb R\), then
\(\lambda_0\in\rho(\Delta_{X,\mc V})\) for every
\(\mc V\in\operatorname{LGr}(V_S,\omega_S)\), since each
\(\Delta_{X,\mc V}\) is self-adjoint. Hence
\(\lambda_0\in\Sigma_X\).
Suppose now that \(\lambda_0\in\mb R\). Then
\(\mc K_X(\lambda_0)\) is isotropic. Choose a Lagrangian subspace
\(\mc W\subset V_S\) containing \(\mc K_X(\lambda_0)\), and choose a
Lagrangian complement \(\mc W^\#\) of \(\mc W\). It follows that
\(
    \mc K_X(\lambda_0)\cap\mc W^\#=\{0\}.
\)
We claim that \(\lambda_0\) is not an eigenvalue of
\(\Delta_{X,\mc W^\#}\). Indeed, suppose that
\(u\in D_{\mc W^\#}\) satisfies
\(
    (\Delta_{X,\mc W^\#}-\lambda_0)u=0.
\)
Then
\(
    \pi_Su
    \in
    \mc K_X(\lambda_0)\cap\mc W^\#
    =
    \{0\}.
\)
Hence \(u\in D_{\min}(\Delta_X)\) and
\(
    (\Delta_{X,\min}-\lambda_0)u=0.
\)
Since
\(\lambda_0\notin\sigma_p(\Delta_{X,\min})\), we conclude that
\(u=0\). Thus
\(\lambda_0\notin\sigma_p(\Delta_{X,\mc W^\#})\).
Since \(\Delta_{X,\mc W^\#}\) is self-adjoint with compact resolvent,
its spectrum consists entirely of isolated eigenvalues of finite
multiplicity. Therefore
\(\lambda_0\in\rho(\Delta_{X,\mc W^\#})\), and hence
\(\lambda_0\in\Sigma_X\). This proves
\[
    \Sigma_X
    =
    \mb C\setminus\sigma_p(\Delta_{X,\min}).
\]
Finally, since \(\Delta_{X,\min}\) is symmetric,
\(\sigma_p(\Delta_{X,\min})\subset\mb R\). Moreover, for any fixed
\(\mc V\in\operatorname{LGr}(V_S,\omega_S)\),
\[
    \sigma_p(\Delta_{X,\min})
    \subset
    \sigma(\Delta_{X,\mc V}).
\]
Since \(\Delta_{X,\mc V}\) has compact resolvent, its spectrum has no
finite accumulation point. Consequently, the point spectrum of
\(\Delta_{X,\min}\) is a closed discrete subset of \(\mb R\). Its
complement in \(\mb C\) is path-connected.
Therefore \(\Sigma_X\) is path-connected.
\end{proof}
Set \(m:=\frac12\dim_{\mb C}V_S\), and let
\(\operatorname{Gr}_m(V_S)\) be the complex Grassmannian. Its points
are the \(m\)-dimensional complex linear subspaces of \(V_S\), and its
complex dimension is \(m^2\).

We regard the Lagrangian Grassmannian
\(\operatorname{LGr}(V_S,\omega_S)\) as a subset of
\(\operatorname{Gr}_m(V_S)\).

For every \(\lambda\in\Sigma_X\), there exists
\(\mc V\in\operatorname{LGr}(V_S,\omega_S)\) such that
\(\lambda\in\rho(\Delta_{X,\mc V})\). Choose a Lagrangian complement
\(\mc V^\#\) of \(\mc V\). The Poisson operator is an isomorphism
from \(\mc V^\#\) onto \(\ker(\Delta_X-\lambda)\). Hence
\(\dim_{\mb C}\ker(\Delta_X-\lambda)=m\).
Moreover, the restriction
\[
    \pi_S\big|_{\ker(\Delta_X-\lambda)}:
    \ker(\Delta_X-\lambda)
    \longrightarrow
    \mc K_X(\lambda)
\]
is a linear isomorphism. Indeed, it is surjective by the definition of
\(\mc K_X(\lambda)\). If
\(u\in\ker(\Delta_X-\lambda)\) and \(\pi_Su=0\), then
\(u\in D_{\min}(\Delta_X)\subset D_{\mc V}\), and hence
\(
    (\Delta_{X,\mc V}-\lambda)u=0.
\)
Since \(\lambda\in\rho(\Delta_{X,\mc V})\), it follows that \(u=0\).
Therefore
\[
    \dim_{\mb C}\mc K_X(\lambda)
    =
    \dim_{\mb C}\ker(\Delta_X-\lambda)
    =
    m.
\]
We thus obtain a map
\[
    \mc K_X:
    \Sigma_X
    \longrightarrow
    \operatorname{Gr}_m(V_S),
    \qquad
    \lambda
    \longmapsto
    \mc K_X(\lambda).
\]
We next prove that \(\mc K_X\) is holomorphic. We first establish the
local graph representation needed for this purpose.

Let \(E\in\operatorname{Gr}_m(V_S)\), and choose
\(F\in\operatorname{Gr}_m(V_S)\) such that
\(V_S=E\oplus F\). Consider the open subset
\[
    \mc U_E
    =
    \left\{
        \mc L\in\operatorname{Gr}_m(V_S):
        \mc L\cap E=\{0\}
    \right\}.
\]
Let \(P_E:V_S\to E\) and \(P_F:V_S\to F\) be the projections
associated with this decomposition.
For every \(\mc L\in\mc U_E\), the restriction
\(P_F|_{\mc L}:\mc L\to F\) is an isomorphism. Define
\(
    A_{\mc L}
    :=
    P_E\circ
    \bigl(P_F|_{\mc L}\bigr)^{-1}
    \in
    \operatorname{Hom}_{\mb C}(F,E).
\)
Then, for every \(h\in F\),
\(
    \bigl(P_F|_{\mc L}\bigr)^{-1}h
    =
    A_{\mc L}h+h,
\)
and hence
\[
    \mc L
    =
    \left\{
        A_{\mc L}h+h:
        h\in F
    \right\}
    =
    \operatorname{graph}_{(E,F)}(A_{\mc L}).
\]
The correspondence
\[
    \kappa_{(E,F)}:
    \mc U_E
    \longrightarrow
    \operatorname{Hom}_{\mb C}(F,E),
    \qquad
    \mc L
    \longmapsto
    A_{\mc L},
\]
defines a coordinate chart on
\(\operatorname{Gr}_m(V_S)\).

\begin{lem}\label{lem:cauchy-data-transverse}
Let \(\mc V\in\operatorname{LGr}(V_S,\omega_S)\). For every
\(\lambda\in\rho(\Delta_{X,\mc V})\),
\[
    \mc K_X(\lambda)\cap\mc V=\{0\}.
\]
\end{lem}

\begin{proof}
Let \(h\in\mc K_X(\lambda)\cap\mc V\). By the definition of
\(\mc K_X(\lambda)\), there exists
\(u_\lambda\in\ker(\Delta_X-\lambda)\) such that
\(\pi_Su_\lambda=h\). Since \(h\in\mc V\), we have
\(u_\lambda\in D_{\mc V}\). Hence
\(
    (\Delta_{X,\mc V}-\lambda)u_\lambda=0.
\)
Since \(\lambda\in\rho(\Delta_{X,\mc V})\),
the operator \(\Delta_{X,\mc V}-\lambda\) is invertible. Therefore
\(u_\lambda=0\), and consequently \(h=\pi_Su_\lambda=0\).
\end{proof}

Let \((\mc V,\mc V^\#)\) be a Lagrangian pair. If
\(\lambda\in\rho(\Delta_{X,\mc V})\), then
\(\mc K_X(\lambda)\cap\mc V=\{0\}\), and hence
\(\mc K_X(\lambda)\in\mc U_{\mc V}\). By the definition of the Weyl
function,
\[
    \mc K_X(\lambda)
    =
    \operatorname{graph}_{(\mc V,\mc V^\#)}
    \bigl(M_{\mc V,\mc V^\#}^X(\lambda)\bigr).
\]
Consequently,
\[
    \kappa_{(\mc V,\mc V^\#)}
    \bigl(\mc K_X(\lambda)\bigr)
    =
    M_{\mc V,\mc V^\#}^X(\lambda)
    \in
    \operatorname{Hom}_{\mb C}(\mc V^\#,\mc V).
\]

\begin{prop}
The map
\(
    \mc K_X:
    \Sigma_X
    \longrightarrow
    \operatorname{Gr}_m(V_S)
\)
is holomorphic.
\end{prop}

\begin{proof}
Let \(\lambda_0\in\Sigma_X\). By the definition of \(\Sigma_X\), there
exists \(\mc V\in\operatorname{LGr}(V_S,\omega_S)\) such that
\(\lambda_0\in\rho(\Delta_{X,\mc V})\). Choose a Lagrangian complement
\(\mc V^\#\) of \(\mc V\).

For every \(\lambda\in\rho(\Delta_{X,\mc V})\), the previous lemma
gives \(\mc K_X(\lambda)\cap\mc V=\{0\}\). Hence
\(
    \mc K_X\bigl(\rho(\Delta_{X,\mc V})\bigr)
    \subset
    \mc U_{\mc V}.
\)
Consider the identity chart on
\(\rho(\Delta_{X,\mc V})\subset\mb C\) and the Grassmannian chart
\(
    \kappa_{(\mc V,\mc V^\#)}:
    \mc U_{\mc V}
    \longrightarrow
    \operatorname{Hom}_{\mb C}(\mc V^\#,\mc V).
\)
The corresponding local coordinate representation of \(\mc K_X\) is
\[
    \kappa_{(\mc V,\mc V^\#)}
    \circ
    \mc K_X\big|_{\rho(\Delta_{X,\mc V})}
    =
    M_{\mc V,\mc V^\#}^X.
\]
Since \(M_{\mc V,\mc V^\#}^X\) is holomorphic on
\(\rho(\Delta_{X,\mc V})\), the map \(\mc K_X\) is holomorphic at
\(\lambda_0\). Since \(\lambda_0\) was arbitrary, \(\mc K_X\) is
holomorphic on \(\Sigma_X\).
\end{proof}

\begin{defn}
The holomorphic map
\[
    \mc K_X:
    \Sigma_X
    \longrightarrow
    \operatorname{Gr}_m(V_S)
\]
is called the Weyl curve of \(X\).
\end{defn}

Differentiating the local coordinate identity established above
immediately gives the following description of the differential of
the Weyl curve.

\begin{prop}\label{prop:differential-weyl-curve}
Let \(\lambda_0\in\Sigma_X\). The holomorphic differential of the
Weyl curve at \(\lambda_0\) is the complex-linear map
\[
    d\mc K_X|_{\lambda_0}:
    T_{\lambda_0}^{1,0}\Sigma_X
    \longrightarrow
    T_{\mc K_X(\lambda_0)}^{1,0}
    \operatorname{Gr}_m(V_S).
\]
Let \((\mc V,\mc V^\#)\) be a Lagrangian pair such that
\(\lambda_0\in\rho(\Delta_{X,\mc V})\). Under the natural
identification
\[
    T_{\mc K_X(\lambda_0)}^{1,0}
    \operatorname{Gr}_m(V_S)
    \simeq
    \operatorname{Hom}_{\mb C}
    \bigl(
        \mc K_X(\lambda_0),
        V_S/\mc K_X(\lambda_0)
    \bigr),
\]
the tangent vector
\(
    d\mc K_X|_{\lambda_0}
    \bigl(
        \left.\frac{\partial}{\partial\lambda}\right|_{\lambda_0}
    \bigr)
\)
is given by
\[
\begin{aligned}
    &d\mc K_X|_{\lambda_0}
    \left(
        \left.\frac{\partial}{\partial\lambda}\right|_{\lambda_0}
    \right)
    \left(
        M_{\mc V,\mc V^\#}^X(\lambda_0)h^\#+h^\#
    \right)
    \\
    &\qquad
    =
    \left[
        \left.
        \frac{d}{d\lambda}
        M_{\mc V,\mc V^\#}^X(\lambda)h^\#
        \right|_{\lambda=\lambda_0}
    \right]
\end{aligned}
\]
for every \(h^\#\in\mc V^\#\), where the brackets denote the class in
\(V_S/\mc K_X(\lambda_0)\).

Equivalently, in the graph-coordinate chart determined by
\((\mc V,\mc V^\#)\),
\[
    d\kappa_{(\mc V,\mc V^\#)}|_{\mc K_X(\lambda_0)}
    \left(
        d\mc K_X|_{\lambda_0}
        \left(
            \left.\frac{\partial}{\partial\lambda}\right|_{\lambda_0}
        \right)
    \right)
    =
    \left.
    \frac{d}{d\lambda}
    M_{\mc V,\mc V^\#}^X(\lambda)
    \right|_{\lambda=\lambda_0}.
\]
\end{prop}

\begin{proof}
Fix \(h^\#\in\mc V^\#\). Near \(\lambda_0\), the vector
\(
    s_{h^\#}(\lambda)
    =
    M_{\mc V,\mc V^\#}^X(\lambda)h^\#+h^\#
\)
belongs to \(\mc K_X(\lambda)\). Differentiating at \(\lambda_0\)
gives
\[
    s_{h^\#}'(\lambda_0)
    =
    \left.
    \frac{d}{d\lambda}
    M_{\mc V,\mc V^\#}^X(\lambda)h^\#
    \right|_{\lambda=\lambda_0}.
\]
Under the standard identification of the holomorphic tangent space of
the Grassmannian with
\(
    \operatorname{Hom}_{\mb C}
    \bigl(
        \mc K_X(\lambda_0),
        V_S/\mc K_X(\lambda_0)
    \bigr),
\)
the corresponding tangent vector is obtained by taking the class of
\(s_{h^\#}'(\lambda_0)\) modulo \(\mc K_X(\lambda_0)\). This proves
the first formula. The coordinate formula follows by differentiating
\(
    \kappa_{(\mc V,\mc V^\#)}\circ\mc K_X
    =
    M_{\mc V,\mc V^\#}^X
\)
at \(\lambda_0\).
\end{proof}

We now study the transition between the local coordinate
representations of the Weyl curve determined by two Lagrangian pairs
\((\mc V,\mc V^\#)\) and \((\mc W,\mc W^\#)\) in
\((V_S,\omega_S)\).

Let \(V,W,E,F\) be subspaces of \(V_S\) such that
\(
    V_S=V\oplus W=E\oplus F.
\)
Every linear map \(T:V_S\to V_S\) admits a block representation
relative to the source decomposition \(V_S=V\oplus W\) and the target
decomposition \(V_S=E\oplus F\). Namely, there exist linear maps
\(
    T_{11}:V\to E,
    T_{12}:W\to E,
    T_{21}:V\to F,
\)
and \(T_{22}:W\to F\) such that, for \(v\in V\) and \(w\in W\),
\[
    T(v+w)
    =
    \bigl(T_{11}v+T_{12}w\bigr)
    +
    \bigl(T_{21}v+T_{22}w\bigr),
\]
where the first term belongs to \(E\) and the second belongs to \(F\).
We denote this block representation by
\[
    [T]_{(V,W)}^{(E,F)}
    =
    \begin{bmatrix}
        T_{11} & T_{12}\\
        T_{21} & T_{22}
    \end{bmatrix}.
\]

Assume that
\(
    \dim_{\mb C}V
    =
    \dim_{\mb C}W
    =
    \dim_{\mb C}E
    =
    \dim_{\mb C}F
    =
    m.
\)
Consider the transition map
\(
    \kappa_{(E,F)}
    \circ
    \kappa_{(V,W)}^{-1}
\)
between the graph-coordinate charts determined by the decompositions
\(V_S=V\oplus W\) and \(V_S=E\oplus F\).
Its domain is the open subset of
\(\operatorname{Hom}_{\mb C}(W,V)\) corresponding to
\(\mc U_V\cap\mc U_E\), and its values lie in
\(\operatorname{Hom}_{\mb C}(F,E)\).

Write
\[
    [\id_{V_S}]_{(V,W)}^{(E,F)}
    =
    \begin{bmatrix}
        A&B\\
        C&D
    \end{bmatrix}.
\]
The standard transition formula for graph coordinates is
\[
    \bigl(
        \kappa_{(E,F)}
        \circ
        \kappa_{(V,W)}^{-1}
    \bigr)(L)
    =
    (AL+B)(CL+D)^{-1}.
\]
Moreover,
\[
    \kappa_{(V,W)}
    \bigl(\mc U_V\cap\mc U_E\bigr)
    =
    \left\{
        L\in\operatorname{Hom}_{\mb C}(W,V):
        CL+D:W\longrightarrow F
        \text{ is invertible}
    \right\}.
\]

We therefore obtain the following transformation law for the local
coordinate representations of the Weyl curve.

\begin{thm}\label{thm:weyl-transformation-law}
Let \((\mc V,\mc V^\#)\) and \((\mc W,\mc W^\#)\) be Lagrangian
pairs. Let
\[
    [\id_{V_S}]_{(\mc V,\mc V^\#)}^{(\mc W,\mc W^\#)}
    =
    \begin{bmatrix}
        A&B\\
        C&D
    \end{bmatrix}
\]
be the block representation of the identity map relative to the two
Lagrangian decompositions. If
\(
    \lambda
    \in
    \rho(\Delta_{X,\mc V})
    \cap
    \rho(\Delta_{X,\mc W}),
\)
then
\(
    C M_{\mc V,\mc V^\#}^X(\lambda)+D:
    \mc V^\#
    \longrightarrow
    \mc W^\#
\)
is invertible, and
\[
    M_{\mc W,\mc W^\#}^X(\lambda)
    =
    \bigl(
        A M_{\mc V,\mc V^\#}^X(\lambda)+B
    \bigr)
    \bigl(
        C M_{\mc V,\mc V^\#}^X(\lambda)+D
    \bigr)^{-1}.
\]
\end{thm}

\begin{proof}
We argue directly from the definition of the Weyl function.

Let \(h^\#\in\mc V^\#\), and set
\(
    u_\lambda
    =
    \Gamma_{\mc V,\mc V^\#}^X(\lambda)h^\#.
\)
Then
\(\widetilde\pi_{\mc V^\#}(u_\lambda)=h^\#\), and
\[
    \pi_S(u_\lambda)
    =
    M_{\mc V,\mc V^\#}^X(\lambda)h^\#+h^\#.
\]
With respect to the decomposition
\(V_S=\mc W\oplus\mc W^\#\), its boundary components are
\[
\begin{aligned}
    \widetilde\pi_{\mc W}(u_\lambda)
    &=
    \bigl(
        A M_{\mc V,\mc V^\#}^X(\lambda)+B
    \bigr)h^\#,
    \\
    \widetilde\pi_{\mc W^\#}(u_\lambda)
    &=
    \bigl(
        C M_{\mc V,\mc V^\#}^X(\lambda)+D
    \bigr)h^\#.
\end{aligned}
\]
Set
\(
    k^\#
    =
    \bigl(
        C M_{\mc V,\mc V^\#}^X(\lambda)+D
    \bigr)h^\#.
\)
Since \(\lambda\in\rho(\Delta_{X,\mc W})\), the defining property of
the Poisson operator implies that
\(
    u_\lambda
    =
    \Gamma_{\mc W,\mc W^\#}^X(\lambda)k^\#.
\)

If \(k^\#=0\), then \(u_\lambda=0\), and hence
\(h^\#=\widetilde\pi_{\mc V^\#}(u_\lambda)=0\). Therefore
\(C M_{\mc V,\mc V^\#}^X(\lambda)+D\) is injective. Since
\(\dim_{\mb C}\mc V^\#=\dim_{\mb C}\mc W^\#=m\), it is invertible.

By the definition of the Weyl function,
\[
    M_{\mc W,\mc W^\#}^X(\lambda)k^\#
    =
    \bigl(
        A M_{\mc V,\mc V^\#}^X(\lambda)+B
    \bigr)h^\#.
\]
Since this holds for every \(h^\#\in\mc V^\#\), we obtain
\[
    M_{\mc W,\mc W^\#}^X(\lambda)
    \bigl(
        C M_{\mc V,\mc V^\#}^X(\lambda)+D
    \bigr)
    =
    A M_{\mc V,\mc V^\#}^X(\lambda)+B.
\]
The asserted transformation formula now follows from the
invertibility established above.
\end{proof}

Set
\(
    \Sigma_X^{\mb R}
    :=
    \Sigma_X\cap\mb R.
\)
For every \(\lambda\in\Sigma_X^{\mb R}\),
Lemma~\ref{lem:cauchy-data-orthogonal-complement} implies that
\(\mc K_X(\lambda)\) is Lagrangian. Hence the restriction
\[
    \mc K_X^{\mb R}
    :=
    \mc K_X|_{\Sigma_X^{\mb R}}:
    \Sigma_X^{\mb R}
    \longrightarrow
    \operatorname{LGr}(V_S,\omega_S)
\]
is smooth. We call \(\mc K_X^{\mb R}\) the real Weyl curve. We now
introduce local graph coordinates on
\(\operatorname{LGr}(V_S,\omega_S)\) and prove that the restriction of
the real Weyl curve to each connected component of
\(\Sigma_X^{\mb R}\) is regular.

We equip \(\operatorname{LGr}(V_S,\omega_S)\) with the subspace
topology induced from \(\operatorname{Gr}_m(V_S)\). Let
\((\mc V,\mc V^\#)\) be a Lagrangian pair, and define
\(
    \mc U_{\mc V}^{\operatorname{Lag}}
    :=
    \mc U_{\mc V}
    \cap
    \operatorname{LGr}(V_S,\omega_S).
\)
Since \(\mc U_{\mc V}\) is open in
\(\operatorname{Gr}_m(V_S)\), the set
\(\mc U_{\mc V}^{\operatorname{Lag}}\) is open in
\(\operatorname{LGr}(V_S,\omega_S)\) with respect to the induced
topology. Equivalently,
\[
    \mc U_{\mc V}^{\operatorname{Lag}}
    =
    \left\{
        \mc L\in\operatorname{LGr}(V_S,\omega_S):
        \mc L\cap\mc V=\{0\}
    \right\}.
\]

For \(M\in\operatorname{Hom}_{\mb C}(\mc V^\#,\mc V)\), its
symplectic adjoint
\(M^\dagger:\mc V^\#\to\mc V\) is defined by
\[
    \omega_S(Mh^\#,k^\#)
    =
    \omega_S(h^\#,M^\dagger k^\#),
    \qquad
    h^\#,k^\#\in\mc V^\#.
\]
The map \(M^\dagger\) is uniquely determined by the nondegeneracy of
the pairing between \(\mc V\) and \(\mc V^\#\) induced by
\(\omega_S\). Moreover,
\[
    \omega_S
    \bigl(
        Mh^\#+h^\#,
        Mk^\#+k^\#
    \bigr)
    =
    \omega_S
    \bigl(
        h^\#,
        (M^\dagger+M)k^\#
    \bigr).
\]
Hence
\(\operatorname{graph}_{(\mc V,\mc V^\#)}(M)\) is Lagrangian if and
only if \(M^\dagger=-M\). Set
\[
    \operatorname{Herm}_{\omega_S}(\mc V^\#,\mc V)
    :=
    \left\{
        M\in\operatorname{Hom}_{\mb C}(\mc V^\#,\mc V):
        M^\dagger=-M
    \right\}.
\]
This is a real vector subspace of
\(\operatorname{Hom}_{\mb C}(\mc V^\#,\mc V)\). The Grassmannian
graph chart determined by \(V_S=\mc V\oplus\mc V^\#\) therefore
restricts to the real coordinate chart
\[
    \kappa_{(\mc V,\mc V^\#)}^{\operatorname{Lag}}:
    \mc U_{\mc V}^{\operatorname{Lag}}
    \longrightarrow
    \operatorname{Herm}_{\omega_S}(\mc V^\#,\mc V),
\]
given by
\(
    \kappa_{(\mc V,\mc V^\#)}^{\operatorname{Lag}}
    \left(
        \operatorname{graph}_{(\mc V,\mc V^\#)}(M)
    \right)
    =
    M.
\)

For every
\(\lambda\in\rho(\Delta_{X,\mc V})\cap\Sigma_X^{\mb R}\), we have
\(
    \mc K_X^{\mb R}(\lambda)
    \in
    \mc U_{\mc V}^{\operatorname{Lag}}
\)
and
\(
    \kappa_{(\mc V,\mc V^\#)}^{\operatorname{Lag}}
    \bigl(\mc K_X^{\mb R}(\lambda)\bigr)
    =
    M_{\mc V,\mc V^\#}^X(\lambda).
\)
Thus the Weyl function is the local coordinate representation of the
real Weyl curve in the Lagrangian graph chart determined by
\((\mc V,\mc V^\#)\). In particular,
\[
    \bigl(M_{\mc V,\mc V^\#}^X(\lambda)\bigr)^\dagger
    =
    -M_{\mc V,\mc V^\#}^X(\lambda).
\]

Let
\(
    \mc L:I\to\operatorname{LGr}(V_S,\omega_S)
\)
be a smooth curve. Its tangent vector at \(\lambda\) is naturally
identified with a Hermitian form
\(\mathfrak q_\lambda\) on \(\mc L(\lambda)\), called the tangent form
of the curve. More explicitly, suppose that, relative to a fixed
Lagrangian pair \((\mc V,\mc V^\#)\), the curve is represented by
\(
    \mc L(\lambda)
    =
    \left\{
        M(\lambda)h^\#+h^\#:
        h^\#\in\mc V^\#
    \right\}.
\)
Then
\[
    \mathfrak q_\lambda\bigl(
        M(\lambda)h^\#+h^\#,
        M(\lambda)k^\#+k^\#
    \bigr)
    =
    \omega_S\bigl(
        \dot M(\lambda)h^\#,
        k^\#
    \bigr)
\]
for \(h^\#,k^\#\in\mc V^\#\). The curve is called positive if
\(\mathfrak q_\lambda(v,v)>0\) for every
\(0\neq v\in\mc L(\lambda)\) and every \(\lambda\in I\).

\begin{prop}[Positivity of the real Weyl curve]
\label{prop:positive-weyl-curve}
Let \((\mc V,\mc V^\#)\) be a Lagrangian pair, and let
\(I\subset\rho(\Delta_{X,\mc V})\cap\Sigma_X^{\mb R}\) be an open
interval. Then the restriction
\[
    \mc K_X^{\mb R}|_I:
    I
    \longrightarrow
    \operatorname{LGr}(V_S,\omega_S)
\]
is a positive curve. Its tangent form satisfies
\[
\mathfrak q_\lambda\bigl(
    M_{\mc V,\mc V^\#}^X(\lambda)h^\#+h^\#,
    M_{\mc V,\mc V^\#}^X(\lambda)k^\#+k^\#
\bigr)
=
\left\langle
    \Gamma_{\mc V,\mc V^\#}^X(\lambda)h^\#,
    \Gamma_{\mc V,\mc V^\#}^X(\lambda)k^\#
\right\rangle_{L^2(X')}
\]
for \(h^\#,k^\#\in\mc V^\#\). In particular,
\[
\mathfrak q_\lambda\bigl(
    M_{\mc V,\mc V^\#}^X(\lambda)h^\#+h^\#,
    M_{\mc V,\mc V^\#}^X(\lambda)h^\#+h^\#
\bigr)
=
\left\|
    \Gamma_{\mc V,\mc V^\#}^X(\lambda)h^\#
\right\|_{L^2(X')}^2
>0
\]
for every \(h^\#\neq0\).
\end{prop}

\begin{proof}
By the definition of the tangent form,
\[
\mathfrak q_\lambda\bigl(
    M_{\mc V,\mc V^\#}^X(\lambda)h^\#+h^\#,
    M_{\mc V,\mc V^\#}^X(\lambda)k^\#+k^\#
\bigr)
=
\omega_S\!\left(
    \frac{d}{d\lambda}
    M_{\mc V,\mc V^\#}^X(\lambda)h^\#,
    k^\#
\right).
\]
By Lemma~\ref{lem:M-derivative-identity},
\[
\omega_S\!\left(
    \frac{d}{d\lambda}
    M_{\mc V,\mc V^\#}^X(\lambda)h^\#,
    k^\#
\right)
=
\left\langle
    \Gamma_{\mc V,\mc V^\#}^X(\lambda)h^\#,
    \Gamma_{\mc V,\mc V^\#}^X(\lambda)k^\#
\right\rangle_{L^2(X')}.
\]
The desired formula follows immediately. Taking
\(k^\#=h^\#\neq0\), and using the fact that
\(\Gamma_{\mc V,\mc V^\#}^X(\lambda)\) is an isomorphism, gives
positivity.
\end{proof}

We next pull back the tautological bundle on
\(\operatorname{Gr}_m(V_S)\) along the Weyl curve. Scott and
Wojciechowski studied determinant lines over Grassmannians of
self-adjoint elliptic boundary conditions and related their canonical
sections to zeta determinants and boundary Fredholm determinants
\cite{ScottWojciechowski1997,ScottWojciechowski2000}; relative zeta
metrics on determinant lines were further considered in
\cite{ScottRelativeZeta}. A related use of the determinant of the
tautological bundle in connection with Schubert divisors appears in
Wang's abstract setting \cite[Section~10.1]{WangWeylCurve}. The
identification made here is specific to the present conic problem: the
base is the spectral-parameter curve \(\mc K_X\), the boundary
determinants arising from the resolvent trace identity become
transition functions of its pullback determinant line, and the
corrected zeta determinants become twisted local representatives.

We recall the local bundle coordinates only to fix the direction of
these transition functions and to identify them precisely with
\(d_{\mc V,\mc W}^X\).

Let \(\mc E\to\operatorname{Gr}_m(V_S)\) be the tautological vector
bundle, whose fiber over
\(\mc L\in\operatorname{Gr}_m(V_S)\) is \(\mc L\) itself. We first
describe its local trivializations and transition functions.

Let \(E\in\operatorname{Gr}_m(V_S)\), and choose
\(F\in\operatorname{Gr}_m(V_S)\) such that \(V_S=E\oplus F\). Recall
that the graph-coordinate chart
\(\kappa_{(E,F)}:\mc U_E\to\operatorname{Hom}_{\mb C}(F,E)\) is
defined on
\(\mc U_E=\{\mc L\in\operatorname{Gr}_m(V_S):\mc L\cap E=\{0\}\}\).
For \(\mc L\in\mc U_E\), write
\(A_{\mc L}=\kappa_{(E,F)}(\mc L)\), so that
\(\mc L=\{A_{\mc L}h+h:h\in F\}\).

Define
\[
    \tau_{(E,F)}:
    \mc U_E\times F
    \longrightarrow
    \mc E|_{\mc U_E},
    \qquad
    \tau_{(E,F)}(\mc L,h)
    =
    (\mc L,A_{\mc L}h+h).
\]
For each \(\mc L\in\mc U_E\), the map
\(h\mapsto A_{\mc L}h+h\) is a linear isomorphism from \(F\) onto
\(\mc L\). Hence \(\tau_{(E,F)}\) is a holomorphic local
trivialization of \(\mc E\) over \(\mc U_E\). Its inverse is given by
\(\tau_{(E,F)}^{-1}(\mc L,\xi)=(\mc L,P_F\xi)\), where
\(P_F:V_S\to F\) is the projection associated with
\(V_S=E\oplus F\).
Now suppose that
\(V_S=V\oplus W=E\oplus F\), and write
\[
    [\id_{V_S}]_{(V,W)}^{(E,F)}
    =
    \begin{bmatrix}
        A&B\\
        C&D
    \end{bmatrix}.
\]
Let \(\mc L\in\mc U_V\cap\mc U_E\), and set
\(L=\kappa_{(V,W)}(\mc L)\). Since
\(\mc L=\{Lh+h:h\in W\}\), the \(F\)-component of \(Lh+h\) is
\((CL+D)h\). Therefore
\[
    \bigl(
        \tau_{(E,F)}^{-1}
        \circ
        \tau_{(V,W)}
    \bigr)(\mc L,h)
    =
    \bigl(\mc L,(CL+D)h\bigr).
\]
Thus the transition function of the tautological vector bundle
\(\mc E\) on \(\mc U_V\cap\mc U_E\) is
\[
    g_{(E,F),(V,W)}:
    \mc U_V\cap\mc U_E
    \longrightarrow
    \operatorname{Iso}_{\mb C}(W,F),
\]
given by
\(
    g_{(E,F),(V,W)}(\mc L)
    =
    C\,\kappa_{(V,W)}(\mc L)+D.
\)
The determinant line bundle of \(\mc E\) is
\(\operatorname{Det}(\mc E):=\bigwedge^m\mc E\), whose fiber over
\(\mc L\in\operatorname{Gr}_m(V_S)\) is
\(\operatorname{Det}(\mc L):=\bigwedge^m\mc L\). Its transition
functions are obtained by taking the top exterior powers of the
transition functions of \(\mc E\). Thus the transition function of
\(\operatorname{Det}(\mc E)\) on \(\mc U_V\cap\mc U_E\) is
\[
    \bigwedge\nolimits^m g_{(E,F),(V,W)}:
    \mc U_V\cap\mc U_E
    \longrightarrow
    \operatorname{Iso}_{\mb C}
    \left(
        \bigwedge\nolimits^m W,
        \bigwedge\nolimits^m F
    \right),
\]
given by
\[
    \left(
        \bigwedge\nolimits^m g_{(E,F),(V,W)}
    \right)(\mc L)
    =
    \bigwedge\nolimits^m
    \left(
        C\,\kappa_{(V,W)}(\mc L)+D
    \right).
\]
After choosing nonzero elements of
\(\bigwedge^m W\) and \(\bigwedge^m F\), this transition function is
represented by the scalar
\(
    \det\!\left(
        C\,\kappa_{(V,W)}(\mc L)+D
    \right).
\)

Pulling back \(\mc E\) along the Weyl curve gives the holomorphic
vector bundle
\(
    \mc E_X:=\mc K_X^*\mc E\to\Sigma_X,
\)
whose fiber at \(\lambda\in\Sigma_X\) is naturally identified with
the boundary-data subspace
\(\mc K_X(\lambda)\subset V_S\).
Since the determinant construction commutes naturally with pullback,
there is a canonical isomorphism
\[
    \operatorname{Det}(\mc E_X)
    =
    \operatorname{Det}(\mc K_X^*\mc E)
    \simeq
    \mc K_X^*\operatorname{Det}(\mc E).
\]
Consequently, the fiber of \(\operatorname{Det}(\mc E_X)\) at
\(\lambda\) is \(\bigwedge^m\mc K_X(\lambda)\).

Whenever several Lagrangian pairs are considered simultaneously, we
choose once and for all a nonzero frame of
\(\bigwedge^m\mc V^\#\) for each pair
\((\mc V,\mc V^\#)\), and use these fixed frames in every pairwise
determinant. With this convention, the resulting scalar transition
functions satisfy the cocycle identities exactly.

Let \((\mc V,\mc V^\#)\) be a Lagrangian pair. Since
\(
    \mc K_X\bigl(\rho(\Delta_{X,\mc V})\bigr)
    \subset\mc U_{\mc V},
\)
the local trivialization
\(\tau_{(\mc V,\mc V^\#)}\) of \(\mc E\) pulls back to the local
trivialization
\[
    \tau_{\mc V,\mc V^\#}^X:
    \rho(\Delta_{X,\mc V})\times\mc V^\#
    \longrightarrow
    \mc E_X|_{\rho(\Delta_{X,\mc V})},
\]
given by
\(
    \tau_{\mc V,\mc V^\#}^X(\lambda,h^\#)
    =
    \bigl(
        \lambda,
        M_{\mc V,\mc V^\#}^X(\lambda)h^\#+h^\#
    \bigr).
\)
In particular, for every fixed \(h^\#\in\mc V^\#\), the map
\(
    s_{h^\#}(\lambda)
    =
    M_{\mc V,\mc V^\#}^X(\lambda)h^\#+h^\#
\)
defines a local holomorphic section of \(\mc E_X\).

Let \((\mc W,\mc W^\#)\) be another Lagrangian pair, and write
\[
    [\id_{V_S}]_{(\mc V,\mc V^\#)}^{(\mc W,\mc W^\#)}
    =
    \begin{bmatrix}
        A&B\\
        C&D
    \end{bmatrix}.
\]
On
\(
    \rho(\Delta_{X,\mc V})
    \cap
    \rho(\Delta_{X,\mc W}),
\)
the transition function of \(\mc E_X\) from the
\((\mc V,\mc V^\#)\)-trivialization to the
\((\mc W,\mc W^\#)\)-trivialization is
\[
    g_{(\mc W,\mc W^\#),(\mc V,\mc V^\#)}^X:
    \rho(\Delta_{X,\mc V})
    \cap
    \rho(\Delta_{X,\mc W})
    \longrightarrow
    \operatorname{Iso}_{\mb C}(\mc V^\#,\mc W^\#),
\]
given by
\(
    g_{(\mc W,\mc W^\#),(\mc V,\mc V^\#)}^X(\lambda)
    =
    C M_{\mc V,\mc V^\#}^X(\lambda)+D.
\)
Taking the top exterior power gives the transition function of
\(\operatorname{Det}(\mc E_X)\):
\[
    \bigwedge\nolimits^m
    g_{(\mc W,\mc W^\#),(\mc V,\mc V^\#)}^X(\lambda)
    =
    \bigwedge\nolimits^m
    \left(
        C M_{\mc V,\mc V^\#}^X(\lambda)+D
    \right).
\]
After choosing nonzero elements of
\(\bigwedge^m\mc V^\#\) and \(\bigwedge^m\mc W^\#\) as local frames,
the corresponding scalar transition function from the
\((\mc V,\mc V^\#)\)-trivialization to the
\((\mc W,\mc W^\#)\)-trivialization is
\[
    g_{(\mc W,\mc W^\#),(\mc V,\mc V^\#)}
    ^{\operatorname{Det},X}(\lambda)
    =
    \det\!\left(
        C M_{\mc V,\mc V^\#}^X(\lambda)+D
    \right).
\]

On the other hand, the transition map in the opposite direction is
\[
\begin{aligned}
    g_{(\mc V,\mc V^\#),(\mc W,\mc W^\#)}^X(\lambda)
    &=
    \left(
        g_{(\mc W,\mc W^\#),(\mc V,\mc V^\#)}^X(\lambda)
    \right)^{-1} \\
    &=
    C_{\mc V}^{\dagger}
    M_{\mc W,\mc W^\#}^X(\lambda)
    +
    A_{\mc V}^{\dagger}.
\end{aligned}
\]
Therefore its scalar determinant is
\[
\begin{aligned}
    g_{(\mc V,\mc V^\#),(\mc W,\mc W^\#)}
    ^{\operatorname{Det},X}(\lambda)
    &=
    \det\!\left(
        C_{\mc V}^{\dagger}
        M_{\mc W,\mc W^\#}^X(\lambda)
        +
        A_{\mc V}^{\dagger}
    \right) \\
    &=
    d_{(\mc V,\mc V^\#),(\mc W,\mc W^\#)}^X(\lambda).
\end{aligned}
\]
Thus the boundary determinant \(d_{\mc V,\mc W}^X\) is precisely the
scalar transition function of the pullback determinant line bundle
\(\operatorname{Det}(\mc E_X)\) from the
\((\mc W,\mc W^\#)\)-trivialization to the
\((\mc V,\mc V^\#)\)-trivialization.

We now restrict the preceding construction to the real Weyl curve.
Set
\(
    \mc E_X^{\mb R}
    :=
    \mc E_X|_{\Sigma_X^{\mb R}}.
\)
Equivalently,
\(
    \mc E_X^{\mb R}
    =
    (\mc K_X^{\mb R})^*
    \bigl(
        \mc E|_{\operatorname{LGr}(V_S,\omega_S)}
    \bigr).
\)
Its fiber at
\(\mu\in\Sigma_X^{\mb R}\) is the Lagrangian boundary-data subspace
\(\mc K_X^{\mb R}(\mu)\). The associated determinant line bundle is
\(
    \operatorname{Det}(\mc E_X^{\mb R})
    =
    \operatorname{Det}(\mc E_X)|_{\Sigma_X^{\mb R}}.
\)
For two Lagrangian pairs
\((\mc V,\mc V^\#)\) and \((\mc W,\mc W^\#)\), the scalar transition
function from the \((\mc W,\mc W^\#)\)-trivialization to the
\((\mc V,\mc V^\#)\)-trivialization on
\(
    \rho(\Delta_{X,\mc V})
    \cap
    \rho(\Delta_{X,\mc W})
    \cap
    \mb R
\)
is \(d_{\mc V,\mc W}^X(\mu)\). Thus the boundary determinants
appearing in the comparison formula are naturally interpreted as the
transition functions of the determinant line bundle induced by the
real Weyl curve.

Let \(I\) be a connected component of \(\Sigma_X^{\mb R}\). Fix a
collection \(\mathfrak A\) of Lagrangian pairs such that the nonempty
sets below cover \(I\), and suppose that
Assumption~\ref{ass:zeta-regularity} holds for every pair of
realizations indexed by \(\mathfrak A\) on each connected component
of their overlap. For \((\mc V,\mc V^\#)\in\mathfrak A\), set
\(
    I_{\mc V}:=I\cap\rho(\Delta_{X,\mc V}).
\)
The nonempty sets \(I_{\mc V}\) form an open cover
\(
\mathfrak U=
\{I_{\mc V}:(\mc V,\mc V^\#)\in\mathfrak A,\
I_{\mc V}\neq\varnothing\}
\)
of \(I\).
For \(\mu\in I_{\mc V}\), define the individual correction factor
\[
    E_{\mc V}^X(\mu)
    :=
    \prod_{\lambda_\ell^{\mc V}<\mu}
    (\mu-\lambda_\ell^{\mc V})^{m(\lambda_\ell^{\mc V})}.
\]
The product is finite, and an empty product is understood to be
\(1\). On \(I_{\mc V}\cap I_{\mc W}\), one then has
\(
    E_{\mc V,\mc W}^X
    =
    E_{\mc V}^X/E_{\mc W}^X.
\)

For pairs in \(\mathfrak A\), on every nonempty overlap
\(I_{\mc V}\cap I_{\mc W}\),
Theorem~\ref{thm:shifted-comparison-formula}, applied on each
connected component of the overlap, gives
\[
    \frac{
        \operatorname{Det}_{\zeta}
        (\Delta_{X,\mc V}-\mu)
    }{
        \operatorname{Det}_{\zeta}
        (\Delta_{X,\mc W}-\mu)
    }
    =
    C_{\mc V,\mc W}(\mu)
    \frac{E_{\mc W}^X(\mu)}{E_{\mc V}^X(\mu)}
    d_{\mc V,\mc W}^X(\mu),
\]
where \(C_{\mc V,\mc W}\) is a locally constant
\(\mb C^\times\)-valued function on the overlap.

Because the fixed determinant frames specified above are used
throughout, on every nonempty triple overlap the multiplicativity of
relative zeta determinants and correction factors, together with the cocycle
identity
\(
d_{\mc V,\mc W}^X d_{\mc W,\mc U}^X
=d_{\mc V,\mc U}^X,
\)
gives
\(
C_{\mc V,\mc U}
=C_{\mc V,\mc W}C_{\mc W,\mc U}.
\)
Thus
\[
    \{C_{\mc V,\mc W}\}
    \in
    \check Z^1
    \bigl(
        \mathfrak U,
        \underline{\mb C^\times}
    \bigr).
\]
This cocycle determines a rank-one complex local system, or
equivalently a flat complex line bundle, \(\mc L_C\to I\). Since
\(I\) is an open interval, \(\mc L_C\) is trivializable, although no
canonical trivialization is specified.

Define
\(
    \widehat D_{\mc V}:I_{\mc V}\to\mb C^\times
\)
by
\(
    \widehat D_{\mc V}(\mu)
    :=
    E_{\mc V}^X(\mu)
    \operatorname{Det}_{\zeta}
    (\Delta_{X,\mc V}-\mu).
\)
The determinant comparison formula gives
\(
    \widehat D_{\mc V}
    =
    C_{\mc V,\mc W}
    d_{\mc V,\mc W}^X
    \widehat D_{\mc W}
\)
on \(I_{\mc V}\cap I_{\mc W}\). Since
\(d_{\mc V,\mc W}^X\) and \(C_{\mc V,\mc W}\) are respectively the
scalar transition functions of
\(\operatorname{Det}(\mc E_X^{\mb R})\) and \(\mc L_C\), the
collection \(\{\widehat D_{\mc V}\}\) satisfies the gluing law for
local representatives of a smooth section of
\[
    \operatorname{Det}(\mc E_X^{\mb R})
    \otimes
    \mc L_C.
\]

The preceding construction yields the following theorem.

\begin{thm}\label{thm:zeta-determinant-section}
Let \(I\) be a connected component of \(\Sigma_X^{\mb R}\), and let
\(\mathfrak A\) be a collection of Lagrangian pairs such that
Assumption~\ref{ass:zeta-regularity} holds for every pair of
realizations indexed by \(\mathfrak A\) on each connected component
of their overlap. For \((\mc V,\mc V^\#)\in\mathfrak A\), set
\(
    I_{\mc V}:=I\cap\rho(\Delta_{X,\mc V}).
\)
Assume that the nonempty sets \(I_{\mc V}\) form an open cover
\(
\mathfrak U=
\{I_{\mc V}:(\mc V,\mc V^\#)\in\mathfrak A,\
I_{\mc V}\neq\varnothing\}
\)
of \(I\), and fix a nonzero frame of
\(\bigwedge^m\mc V^\#\) for every pair in \(\mathfrak A\).
On every nonempty overlap \(I_{\mc V}\cap I_{\mc W}\), let
\(C_{\mc V,\mc W}\) be the locally constant
\(\mb C^\times\)-valued function obtained by applying
Theorem~\ref{thm:shifted-comparison-formula} to each connected
component of the overlap. Then
\(
    \{C_{\mc V,\mc W}\}
    \in
    \check Z^1(\mathfrak U,\underline{\mb C^\times})
\)
and hence determines a rank-one complex local system
\(\mc L_C\to I\). Define
\(
    \widehat D_{\mc V}:I_{\mc V}\to\mb C^\times
\)
by
\(
    \widehat D_{\mc V}(\mu)
    :=
    E_{\mc V}^X(\mu)
    \operatorname{Det}_{\zeta}
    (\Delta_{X,\mc V}-\mu).
\)
Then the collection \(\{\widehat D_{\mc V}\}\) defines a smooth
section of
\(
    \operatorname{Det}(\mc E_X^{\mb R})
    \otimes
    \mc L_C
\)
over \(I\).
\end{thm}

\section{Conclusion}

We have constructed the boundary symplectic space of a radial conic
Laplacian explicitly from its critical asymptotic coefficients and
identified it with the quotient of the maximal domain by the minimal
domain. This normalization places
the classification of self-adjoint realizations, the Poisson and Weyl
operators, Kreĭn's resolvent formula, and the resolvent trace identity
in one finite-dimensional boundary framework. Under
Assumption~\ref{ass:zeta-regularity}, the same framework gives the
positive-spectrum zeta-determinant comparison, including the finite
correction from the omitted nonpositive spectrum.

The geometric content is encoded by the holomorphic Weyl curve
\(\mc K_X:\Sigma_X\to\operatorname{Gr}_m(V_S)\). The Weyl functions
are precisely its local graph coordinates. On the real locus, its
tangent form is the \(L^2(X')\)-pairing of the corresponding formal
solutions and is therefore positive. Thus the derivative identity for
the Weyl function is the coordinate expression of an intrinsic
geometric statement. Likewise, the boundary determinants
\(d_{\mc V,\mc W}^X\) are transition functions of the pullback
determinant line, while the corrected zeta determinants are local
representatives of a section of
\(\operatorname{Det}(\mc E_X^{\mb R})\otimes\mc L_C\).

The explicit Fourier analysis used to construct the critical boundary
space is specific to the present two-dimensional setting. The standing
zeta hypothesis is not asserted here for arbitrary self-adjoint cone
realizations. Once an analogous boundary space, Green form, and the
required spectral regularity have been established, however, the
subsequent Weyl-curve, resolvent, and determinant arguments apply in
the same finite-dimensional form.

\section*{Acknowledgments}

This work was developed over several years and was supported in part by the
Ministry of Science and Technology of Taiwan (now the National Science and
Technology Council) under Grant Nos. 109-2115-M-006-012,
110-2115-M-006-007, and 111-2115-M-006-018. The author is deeply grateful to
his family for their unwavering support during difficult times.


\begin{thebibliography}{99}\large

\bibitem{bd}
D. Burghelea, L. Friedlander and T. Kappeler,
Mayer--Vietoris type formula for determinants of elliptic differential
operators,
J. Funct. Anal. 107 (1992), no. 1, 34--65.

\bibitem{bjlm}
J. Behrndt and M. Langer,
Elliptic operators, Dirichlet-to-Neumann maps and quasi boundary triples,
in \emph{Operator Methods for Boundary Value Problems}, 121--160,
London Math. Soc. Lecture Note Ser., 404,
Cambridge University Press, Cambridge, 2012.

\bibitem{bjlm2}
J. Behrndt, M. Langer and V. Lotoreichik,
Trace formulae and singular values of resolvent power differences of
self-adjoint elliptic operators,
J. Lond. Math. Soc. (2) 88 (2013), no. 2, 319--337.

\bibitem{cj}
J. Cheeger,
On the spectral geometry of spaces with cone-like singularities,
Proc. Natl. Acad. Sci. USA 76 (1979), no. 5, 2103--2106.

\bibitem{cj2}
J. Cheeger,
Spectral geometry of singular Riemannian spaces,
J. Differential Geom. 18 (1983), no. 4, 575--657.

\bibitem{dva}
V. A. Derkach and M. M. Malamud,
Generalized resolvents and the boundary value problems for Hermitian operators
with gaps,
J. Funct. Anal. 95 (1991), no. 1, 1--95.

\bibitem{FR}
R. Forman,
Functional determinants and geometry,
Invent. Math. 88 (1987), no. 3, 447--493.

\bibitem{gjb}
J. B. Gil and G. A. Mendoza,
Adjoints of elliptic cone operators,
Amer. J. Math. 125 (2003), no. 2, 357--408.

\bibitem{GilKrainerMendoza2007}
J. B. Gil, T. Krainer and G. A. Mendoza,
Geometry and spectra of closed extensions of elliptic cone operators,
Canad. J. Math. 59 (2007), no. 4, 742--794.

\bibitem{HL}
L. Hillairet and A. Kokotov,
Krein formula and S-matrix for Euclidean surfaces with conical singularities,
J. Geom. Anal. 23 (2013), no. 3, 1498--1529.

\bibitem{LP}
P. Loya, P. McDonald and J. Park,
Zeta regularized determinants for conic manifolds,
J. Funct. Anal. 242 (2007), no. 1, 195--229.

\bibitem{kk}
K. Kirsten, P. Loya and J. Park,
Functional determinants for general self-adjoint extensions of Laplace-type
operators resulting from the generalized cone,
Manuscripta Math. 125 (2008), no. 1, 95--126.

\bibitem{KirstenLoyaParkExotic}
K. Kirsten, P. Loya and J. Park,
Exotic expansions and pathological properties of zeta-functions on conic
manifolds, with an appendix by B. Vertman,
J. Geom. Anal. 18 (2008), no. 3, 835--888.

\bibitem{lm}
M. Lesch,
\emph{Operators of Fuchs Type, Conical Singularities, and Asymptotic Methods},
Teubner-Texte Math., 136,
B. G. Teubner Verlagsgesellschaft mbH, Stuttgart, 1997.

\bibitem{WH}
H. Weyl,
{\"U}ber gew{\"o}hnliche Differentialgleichungen mit Singularit{\"a}ten und die
zugeh{\"o}rigen Entwicklungen willk{\"u}rlicher Funktionen,
Math. Ann. 68 (1910), no. 2, 220--269.

\bibitem{MW}
W. M{\"u}ller,
Relative zeta functions, relative determinants and scattering theory,
Comm. Math. Phys. 192 (1998), no. 2, 309--347.

\bibitem{SBW}
B.-W. Schulze,
\emph{Pseudo-Differential Operators on Manifolds with Singularities},
Stud. Math. Appl., 24,
North-Holland Publishing Co., Amsterdam, 1991.

\bibitem{ScottRelativeZeta}
S. Scott,
Relative zeta determinants and the geometry of the determinant line bundle,
Electron. Res. Announc. Amer. Math. Soc. 7 (2001), 8--16.

\bibitem{ScottWojciechowski1997}
S. G. Scott and K. P. Wojciechowski,
Determinants, Grassmannians and elliptic boundary value problems for the
Dirac operator,
Lett. Math. Phys. 40 (1997), no. 2, 135--145.

\bibitem{ScottWojciechowski2000}
S. G. Scott and K. P. Wojciechowski,
The \(\zeta\)-determinant and Quillen determinant for a Dirac operator on a
manifold with boundary,
Geom. Funct. Anal. 10 (2000), no. 5, 1202--1236.

\bibitem{vnj}
J. von Neumann,
Allgemeine Eigenwerttheorie Hermitescher Funktionaloperatoren,
Math. Ann. 102 (1930), no. 1, 49--131.

\bibitem{WangWeylCurve}
Yicao Wang,
Complex analysis of symmetric operators, I,
arXiv:2408.10968v2 [math.FA], 2024.
\end{thebibliography}
\end{document}